\documentclass[11pt]{amsart}
\usepackage{mathrsfs}
\usepackage{amsfonts}
\usepackage{latexsym,amsmath,amssymb}
 \renewcommand{\epsilon}{\varepsilon}

\newtheorem{theorem}{Theorem}[section]
 
 \newtheorem{lemma}[theorem]{Lemma}

 \newtheorem{proposition}[theorem]{proposition}
 \newtheorem{Proposition}[theorem]{Proposition}
\newtheorem{deff}[theorem]{Definition}
 \newtheorem{rem}[theorem]{Remark}
 \newcommand{\bth}{\begin{theorem}}
 \newcommand{\ble}{\begin{lemma}}
 \newcommand{\bcor}{\begin{corr}}
 \newcommand{\bdeff}{\begin{deff}}
 \newcommand{\bprop}{\begin{proposition}}
 \newcommand{\ele}{\end{lemma}}
 \newcommand{\ecor}{\end{corr}}
 \newcommand{\edeff}{\end{deff}}
 
 \newcommand{\eprop}{\end{proposition}}

 \renewcommand{\Pi}{\varPi}

 \renewcommand{\epsilon}{\varepsilon}

\numberwithin{equation}{section}

\title
[Nonlinear second boundary problem]{On the Second Boundary Value Problem for a Class of Fully Nonlinear Flow III}

\thanks{The second author is supported by National Natural Science Foundation of China (No. 11771103 and 11871102) and Guangxi Natural Science Foundation (2017GXNSFFA198017). The third author is supported in part by the National Natural Science Foundation of China (11631002 and 11871102).}

\date{}

\begin{document}
\maketitle

\begin{center}Chong Wang \footnote{Fengtai School of the High School Affiliated to Renmin University of China, Beijing 100074, China. wch0229@mail.bnu.edu.cn}
$\cdot$ Rongli Huang \footnote{The corresponding author. School of Mathematics and Statistics, Guangxi Normal University, Guangxi 541004, China. ronglihuangmath@gxnu.edu.cn}
$\cdot$ Jiguang Bao \footnote{School of Mathematical Sciences, Beijing Normal University, Laboratory of Mathematics and Complex Systems, Ministry of Education, Beijing 100875, China. jgbao@bnu.edu.cn}
\end{center}

%\section{}
%\subsection{}

\begin{abstract}
 We study the solvability of the second boundary value problem of the Lagrangian mean curvature equation arising from  special Lagrangian geometry.  By the parabolic method we obtain the existence and uniqueness of the smooth uniformly convex solution, which generalizes the Brendle-Warren's theorem about minimal Lagrangian diffeomorphism in Euclidean metric space.
\end{abstract}


%{\bfseries Key words:}\quad Second boundary problem; Fully nonlinear equation; Minimal Lagrangian graph.

%\let\thefootnote\relax\footnote{2010 \textit{Mathematics Subject Classification: Primary 53C44; Secondary 53A10.}

%\let\thefootnote\relax\footnote{\textit{Keywords and phrases: Second boundary problem; Fully nonlinear equation; Minimal Lagrangian graph.}

\section{Introduction}
In this work, we are interested in  the long time existence and convergence of convex solutions for special variables, which solves the fully nonlinear equation
\begin{equation}\label{eeee1.1}
\frac{\partial u}{\partial t}=F\left(\lambda (D^2u)\right)-f(x),\ \ t>0,\ x\in\Omega,
\end{equation}
associated with the second boundary value condition
\begin{equation}\label{eeee1.2}
Du(\Omega)=\tilde{\Omega},  \ \ t>0,
\end{equation}
and the initial condition
\begin{equation}\label{eeee1.3}
u=u_{0}, \ \ t=0,\ x\in\Omega
\end{equation}
for given $F$, $f$ and $u_{0}$, where $\Omega$ and $\tilde{\Omega}$ are two uniformly convex bounded domains with smooth boundary in $\mathbb{R}^{n}$ and $\lambda(D^2 u)=(\lambda_1,\cdots, \lambda_n)$ are the eigenvalues of Hessian matrix $D^2 u$. One of our main goal to study the flow is to obtain the existence and uniqueness of the smooth uniformly convex solution for the second boundary value problem of the Lagrangian mean curvature equation
\begin{equation}\label{e1.1}
\left\{ \begin{aligned}F_\tau(\lambda(D^2 u))&=\kappa\cdot x+c,\ \ x\in \Omega, \\
Du(\Omega)&=\tilde{\Omega},
\end{aligned} \right.
\end{equation}
where $\kappa\in \mathbb{R}^{n}$ is a constant vector, $c$ is a constant to be determined and
\begin{equation}\label{e1.101}
F_{\tau}(\lambda):=\left\{ \begin{aligned}
&\frac{1}{n}\sum_{i=1}^n\ln\lambda_{i}, &&\tau=0, \\
& \frac{\sqrt{a^2+1}}{2b}\sum_{i=1}^n\ln\frac{\lambda_{i}+a-b}{\lambda_{i}+a+b},  &&0<\tau<\frac{\pi}{4},\\
& -\sqrt{2}\sum_{i=1}^n\frac{1}{1+\lambda_{i}}, &&\tau=\frac{\pi}{4},\\
& \frac{\sqrt{a^2+1}}{b}\sum_{i=1}^n\arctan\frac{\lambda_{i}+a-b}{\lambda_{i}+a+b},  \ \ &&\frac{\pi}{4}<\tau<\frac{\pi}{2},\\
& \sum_{i=1}^n\arctan\lambda_{i}, &&\tau=\frac{\pi}{2},
\end{aligned} \right.
\end{equation}
where $a=\cot \tau$, $b=\sqrt{|\cot^2\tau-1|}$. Regarding the equation, the details can be seen in \cite{WHB}.

Let
$$g_\tau=\sin \tau \delta_0+\cos \tau g_0,\ \ \tau\in\left[0,\frac{\pi}{2}\right]$$
be the linear combined metric of the standard Euclidean metric
$$\delta_0=\sum_{i=1}^n dx_i\otimes dx_i+\sum_{j=1}^ndy_j\otimes dy_j$$
and the pseudo-Euclidean metric
$$g_{0}=\frac{1}{2}\sum_{i=1}^n dx_i\otimes dy_i+\frac{1}{2}\sum_{j=1}^n dy_j\otimes dx_j$$
in $\mathbb{R}^{n}\times \mathbb{R}^{n}$.

Under  the framework of calibrated geometry in $(\mathbb{R}^n\times\mathbb{R}^n, g_\tau)$, Warren \cite{MW} firstly obtained the special Lagrangian equation
as the form
\begin{equation}\label{equ0.1.5}
F_\tau(\lambda(D^2 u))=c,
\end{equation}
which is a special case of (\ref{e1.1}) when $\kappa\equiv 0$. Then $(x, Du(x))$ is a minimal Lagrangian graph in $(\mathbb{R}^n\times\mathbb{R}^n, g_\tau)$.

If $\tau=0$, (\ref{equ0.1.5}) becomes the famous Monge-Amp\`{e}re equation
$$\det D^2u=e^{2c},$$
which the general form is
\begin{equation}\label{equ0.1.3}
\det D^2 u=f(x, u, Du).
\end{equation}

As for $\tau=\frac{\pi}{2}$, one can show  that (\ref{equ0.1.5}) is  the classical special Lagrangian equation
\begin{equation}\label{equ0.1.1}
\sum_{i=1}^{n}\arctan\lambda_{i}(D^2 u)=c.
\end{equation}
The special Lagrangian equation (\ref{equ0.1.1}) was first introduced by Harvey and Lawson in \cite{HL} back in 1982. Its solutions $u$ were shown to have the property that the graph $(x, Du(x))$ in $(\mathbb{R}^n\times\mathbb{R}^n, \delta_0)$ is a Lagrangian submanifold which is absolutely
volume-minimizing, and the linearization at any solution is elliptic. They proved that a Lagrangian graph $(x, Du(x))$ in $(\mathbb{R}^n\times\mathbb{R}^n, \delta_0)$ is minimal if and only if the Lagrangian angle is a constant, that is, (\ref{equ0.1.1}) holds.
Interestingly, several methods for studying the  Bernstein type theorems occured  in the literature \cite{JX} and  \cite{Y}. Jost and Xin \cite{JX} used the properties of harmonic maps into convex subsets of Grassmannians. Yuan \cite{Y} showed that entire convex solutions of (\ref{equ0.1.1})
must be  a quadratic polynomial based on the geometric measure theory.

People have worked on showing the existence of the minimal  Lagrangian graphs ($\kappa\equiv 0$) and
$Du$ is a diffeomorphism from $\Omega$ to $\tilde{\Omega}$. That is,
\begin{equation}\label{e1.1.1}
\left\{ \begin{aligned}F_\tau(\lambda(D^2 u))&=c,\ \ x\in \Omega, \\
Du(\Omega)&=\tilde{\Omega}.
\end{aligned} \right.
\end{equation}
Here $Du$ is a minimal Lagrangian diffeomorphism from $\Omega$ to $\tilde \Omega$. In the case of $\tau=0$, in dimension 2, Delano\"{e} \cite{P} obtained a unique smooth solution for the second boundary value problem of the Monge-Amp\`{e}re equation if both domains are uniformly convex. Later the generalization of Delano\"{e}'s theorem to higher dimensions was given by Caffarelli \cite{L} and Urbas \cite{JU}. Using the parabolic method, Schn\"{u}rer and Smoczyk \cite{OK} also obtained the existence of solutions to (\ref{e1.1.1}). As far as $\tau=\frac{\pi}{2}$ is concerned, Brendle and Warren \cite{SM} proved the existence and uniqueness of the solution by the elliptic method, and the second author \cite{HR} obtained the existence of solution by considering the second boundary value problem for Lagrangian mean curvature flow. Then by the elliptic and parabolic method, the second author with Ou \cite{HO}, Ye \cite{HRY}  and Chen \cite{CHY} proved the existence and uniqueness of the solution for $0<\tau<\frac{\pi}{2}$.

We are now in a position to find out the Lagrangian graph $(x, Du(x))$ prescribed constant mean curvature vector $\kappa$ in $(\mathbb{R}^n\times\mathbb{R}^n, g_\tau)$ such that $Du$ is the diffeomorphism between two uniformly convex bounded domains. Thus it can be described by the equation (\ref{e1.1}), see  \cite{WHB}.

 By the continuity method, it follows from our early work  \cite{WHB}  we obtain the existence and uniqueness of the smooth uniformly convex solution to (\ref{e1.1}). That is,

\begin{theorem}\label{t1.1}
For $\tau\in\left(0,\frac{\pi}{2}\right]$, if $|\kappa|$ is sufficiently small, then there exist a uniformly convex solution $u\in C^{\infty}(\bar{\Omega})$ and a unique constant $c$ solving (\ref{e1.1}), and $u$ is unique up to a constant.
\end{theorem}
Theorem \ref{t1.1} exhibits  an extension of the previous work on $\kappa=0$ done by Brendle-Warren \cite{SM}, Huang \cite{HR}, Huang-Ou \cite{HO}, Huang-Ye \cite{HRY} and Chen-Huang-Ye \cite{CHY}.

In the present paper, we pursue a strategy of deriving  asymptotic convergence theorem to the solutions of (\ref{eeee1.1})-(\ref{eeee1.3})
for proving Theorem \ref{t1.1} based purely on the previous results of S.J. Altschuler and  L.F. Wu \cite{SL}, O.C. Schn\"{u}rer \cite{OC}, J. Kitagawa \cite{JK}.

Motivated by the work of Huang-Ou \cite{HO} and Huang-Ye \cite{HRY}, we introduce a class of nonlinear functions containing $F_\tau(\lambda)$, $\tau\in (0,\frac{\pi}{2}]$.

For $0<\alpha_{0}<1$, let $F(\lambda_{1},\cdots, \lambda_{n})$ be a $C^{2+\alpha_{0}}$ symmetric function defined on
$$\Gamma^+_n:=\left\{(\lambda_1,\cdots,\lambda_n)\in \mathbb{R}^n:\lambda_i>0,\ i=1,\cdots,n\right\},$$
and satisfy
\begin{equation}\label{e1.2.0}
-\infty<F(0,\cdots,0)<F(+\infty,\cdots,+\infty)<+\infty,
\end{equation}
\begin{equation}\label{e1.2}
\frac{\partial F}{\partial \lambda_i}>0,\ \ 1\leq i\leq n\ \  \text{on}\ \  \Gamma^+_n,
\end{equation}
and
\begin{equation}\label{e1.2.1}
\left(\frac{\partial^2 F}{\partial \lambda_i\partial \lambda_j}\right)\leq 0\ \  \text{on}\ \  \Gamma^+_n.
\end{equation}
For any $(\mu_1,\cdots,\mu_n)\in {\Gamma}^+_n$, denote
$$\lambda_i=\frac{1}{\mu_i},\ \ 1\leq i\leq n,$$
and
$$\tilde F (\mu_1,\cdots,\mu_n):=-F(\lambda_1,\cdots,\lambda_n).$$
Assume that
\begin{equation}\label{e1.2.2}
\left(\frac{\partial^2 \tilde F}{\partial \mu_i\partial \mu_j}\right)\leq 0\ \  \text{on}\ \  \Gamma^+_n.
\end{equation}
For any $s_1>0$, $s_2>0$, define
$$\Gamma^{+}_{]s_1,s_2[}=\{(\lambda_{1},\cdots, \lambda_{n})\in {\Gamma}^+_n:0\leq\min_{1\leq i\leq n}\lambda_i\leq s_1,\ \max_{1\leq i\leq n}\lambda_i\geq s_2\}.$$
We assume that there exist positive constants $\Lambda_1$ and $\Lambda_2$, depending  on $s_1$ and $s_2$, such that for any $(\lambda_{1},\cdots, \lambda_{n})\in \Gamma^{+}_{]s_1,s_2[}$,
\begin{equation}\label{e1.3}
 \Lambda_1\leq\sum^{n}_{i=1}\frac{\partial F}{\partial \lambda_{i}}\leq \Lambda_2,
\end{equation}
and
\begin{equation}\label{e1.4}
  \Lambda_1\leq\sum^{n}_{i=1}\frac{\partial F}{\partial \lambda_{i}}\lambda^{2}_{i}\leq \Lambda_2.
\end{equation}
\begin{rem}
Since
$$\frac{\partial^2 \tilde F}{\partial \mu_i\partial \mu_j}=-\frac{\partial^2 F}{\partial \lambda_i\partial \lambda_j}\lambda_i^2\lambda_j^2-2\lambda_i^3\delta_{ij}\frac{\partial F}{\partial \lambda_i},$$
we cannot deduce (\ref{e1.2.2}) from (\ref{e1.2}) and (\ref{e1.2.1}).
\end{rem}
For $f(x)\in C^{2+\alpha_{0}}(\bar{\Omega})$, we define
$$\mathop{\operatorname{osc}}_{\bar{\Omega}}(f):=\max_{x,y\in\bar{\Omega}}|f(x)-f(y)|,$$
and
$${\mathscr{A}}_\delta:=\left\{f(x)\in C^{2+\alpha_{0}}(\bar{\Omega}): f\ \text{is concave},\ \mathop{\operatorname{osc}}_{\bar{\Omega}}(f) \leq\delta \right\}.$$
 The constant
 $\delta$ is any positive constant satisfying
 $$\delta<\min\left\{F(+\infty,\cdots,+\infty)-\max_{\bar \Omega}F\left(\lambda (D^2 u_0)\right),\min_{\bar \Omega}F\left(\lambda(D^2 u_0)\right)-F(0,\cdots,0)\right\}.$$

\begin{rem}\label{rem2020}
Let $f(x)=\kappa\cdot x$ and if $|\kappa|$ is sufficiently small, then  $f(x)\in {\mathscr{A}}_\delta$.
\end{rem}

Our main results are the following:
\begin{theorem}\label{tttt1.1}
Let $F$ satisfy the structure conditions (\ref{e1.2.0})-(\ref{e1.4}) and $f\in {\mathscr{A}}_\delta$. If $|Df|$ is sufficiently small, then for any given initial function $u_{0}$ which is uniformly convex and satisfies $Du_{0}(\Omega)=\tilde{\Omega}$, the uniformly convex solution of (\ref{eeee1.1})-(\ref{eeee1.3}) exists for all $t\geq 0$ and $u(\cdot,t)$ converges to a function $u^{\infty}(x,t)=u_\infty(x)+c_{\infty}\cdot t$ in $C^{1+\zeta}(\bar{\Omega})\cap C^{4+\alpha}(\bar{D})$ as $t\rightarrow\infty$ for any $D\subset\subset\Omega$, $0<\zeta<1$ and $0<\alpha<\alpha_{0}$. That is,
$$\lim_{t\rightarrow+\infty}\|u(\cdot,t)-u^{\infty}(\cdot,t)\|_{C^{1+\zeta}(\bar{\Omega})}=0,\quad
  \lim_{t\rightarrow+\infty}\|u(\cdot,t)-u^{\infty}(\cdot,t)\|_{C^{4+\alpha}(\bar{D})}=0.$$
And $u_{\infty}(x)\in C^{1+1}(\bar{\Omega})\cap C^{4+\alpha_0}(\Omega)$ is a solution of
\begin{equation}\label{eeee1.18}
\left\{ \begin{aligned}
F\left(\lambda (D^2u)\right)&=f(x)+c_{\infty},\ \ x\in \Omega, \\
Du(\Omega)&=\tilde{\Omega}.
\end{aligned} \right.
\end{equation}
The constant $c_{\infty}$ depends only on $\Omega$, $\tilde{\Omega}$, $u_0$, $f$, $\delta$ and $F$. The solution to (\ref{eeee1.18}) is unique up to additions of constants.

Especially, if $F$ and $f$ are smooth, then there exist a uniformly convex solution $u_{\infty}(x)\in C^{\infty}(\bar{\Omega})$ and a constant $c_{\infty}$ solving (\ref{eeee1.18}).
\end{theorem}

\begin{rem}\label{re1.1}
As is shown in the Remark 2.5 of \cite{WHB}, the oscillation condition  of $f(x)$ can not be omitted.
\end{rem}

The rest of this article is organized as follows. The next section is to present the structure condition for the operator $F_\tau$
and then we can exhibit that Theorem \ref{t1.1} is a corollary of Theorem \ref{tttt1.1}. To prove the main theorem, we verify the short time existence of the parabolic flow in Section 3. Thus Section 4 is devoted to carry out the strictly oblique estimate and the $C^2$ estimate. Eventually, we give the long time existence and convergence of the parabolic flow in Section 5.

Throughout the following, Einstein's convention of summation over repeated indices will be adopted. We denote, for a smooth function $u$,
 $$u_{i}=\dfrac{\partial u}{\partial x_{i}},\ u_{ij}=\dfrac{\partial^{2}u}{\partial x_{i}\partial x_{j}},\ u_{ijk}=\dfrac{\partial^{3}u}{\partial x_{i}\partial x_{j}\partial
x_{k}}, \cdots .$$

\section{ Preliminary step of Theorem \ref{t1.1}}

In the following we are going to describe the analytic structure of the operator $F_\tau$ by direct computation.

It is obvious that $F_\tau(\lambda_1, \cdots, \lambda_n)$, $\tau\in\left(0,\frac{\pi}{2}\right]$ is a smooth symmetric function defined on ${\Gamma}^+_n$. For technical reasons, it is necessary to push further the calculation and we get
\begin{equation*}
F_\tau(0,\cdots,0)=\left\{ \begin{aligned}
& \frac{n\sqrt{a^2+1}}{2b}\ln\frac{a-b}{a+b},  &&0<\tau<\frac{\pi}{4},\\
& -\sqrt{2}n, &&\tau=\frac{\pi}{4},\\
& \frac{n\sqrt{a^2+1}}{b}\arctan\frac{a-b}{a+b}, \quad \quad &&\frac{\pi}{4}<\tau<\frac{\pi}{2},\\
& 0, \,\,&&\tau=\frac{\pi}{2},
\end{aligned} \right.
\end{equation*}
\begin{equation*}
F_\tau(+\infty,\cdots,+\infty)=\left\{ \begin{aligned}
& 0,  &&0<\tau<\frac{\pi}{4},\\
& 0, &&\tau=\frac{\pi}{4},\\
& \frac{n\pi\sqrt{a^2+1}}{4b}, \quad \quad &&\frac{\pi}{4}<\tau<\frac{\pi}{2},\\
& \frac{n\pi}{2}, \,\,&&\tau=\frac{\pi}{2},
\end{aligned} \right.
\end{equation*}
\begin{equation*}
\frac{\partial F_{\tau}}{\partial \lambda_i}=\left\{ \begin{aligned}
& \frac{\sqrt{a^2+1}}{(\lambda_i+a)^2-b^2}, \ \ \ \  &&0<\tau<\frac{\pi}{4},\\
& \frac{\sqrt{2}}{(1+\lambda_{i})^2}, &&\tau=\frac{\pi}{4},\\
& \frac{\sqrt{a^2+1}}{(\lambda_i+a)^2+b^2}, &&\frac{\pi}{4}<\tau<\frac{\pi}{2},\\
& \frac{1}{1+\lambda^2_{i}},  &&\tau=\frac{\pi}{2},
\end{aligned} \right.
\end{equation*}
and
\begin{equation*}
\frac{\partial^2 F_{\tau}}{\partial \lambda_i \partial \lambda_j}=\left\{ \begin{aligned}
& -\frac{2\sqrt{a^2+1}(\lambda_j+a)\delta_{ij}}{\left[(\lambda_i+a)^2-b^2\right]^2}, \ \ \ \  &&0<\tau<\frac{\pi}{4},\\
& -\frac{2\sqrt{2}\delta_{ij}}{(1+\lambda_{i})^3}, &&\tau=\frac{\pi}{4},\\
& -\frac{2\sqrt{a^2+1}(\lambda_j+a)\delta_{ij}}{\left[(\lambda_i+a)^2+b^2\right]^2}, &&\frac{\pi}{4}<\tau<\frac{\pi}{2},\\
& -\frac{2\lambda_j\delta_{ij}}{\left(1+\lambda^2_{i}\right)^2},  &&\tau=\frac{\pi}{2},
\end{aligned} \right.
\end{equation*}
for $i,j=1,\cdots,n$. Then
\begin{equation}\label{e2.1.1.1}
-\infty<F_\tau(0,\cdots,0)< F_\tau(+\infty,\cdots,+\infty)<+\infty,\ \ \tau\in\left(0,\frac{\pi}{2}\right],
\end{equation}
\begin{equation}\label{e2.1.2}
\frac{\partial F_\tau}{\partial \lambda_i}>0,\ \ 1\leq i\leq n \ \  \text{on}\ \  \Gamma^+_n,
\end{equation}
and
\begin{equation}\label{e2.1.3}
\left(\frac{\partial^2 F_\tau}{\partial \lambda_i\partial \lambda_j}\right)\leq 0 \ \  \text{on}\ \ \Gamma^+_n.
\end{equation}

Then for any $(\lambda_{1},\cdots, \lambda_{n})\in \Gamma^{+}_{]s_{1},s_{2}[}$, we have
\begin{equation}\label{e2.1.4}
\sum_{i=1}^n\frac{\partial F_{\tau}}{\partial \lambda_i}\in\left\{ \begin{aligned}
&\left[\frac{\sqrt{a^2+1}}{(s_{1}+a)^2-b^2}, \frac{n\sqrt{a^2+1}}{a^2-b^2}\right], \ \ \ \  &&0<\tau<\frac{\pi}{4},\\
& \left[\frac{\sqrt{2}}{(1+s_{1})^2},n\sqrt{2}\right], &&\tau=\frac{\pi}{4},\\
& \left[\frac{\sqrt{a^2+1}}{(s_{1}+a)^2+b^2},\frac{n\sqrt{a^2+1}}{a^2+b^2}\right], &&\frac{\pi}{4}<\tau<\frac{\pi}{2},\\
&\left[\frac{1}{1+s_{1}^2} ,n\right],  &&\tau=\frac{\pi}{2},
\end{aligned} \right.
\end{equation}
and
\begin{equation}\label{e2.1.6}
\sum_{i=1}^n\frac{\partial F_{\tau}}{\partial \lambda_i}\lambda_i^2\in\left\{ \begin{aligned}
&\left[\frac{s_2^2\sqrt{a^2+1}}{(s_2+a)^2-b^2}, n\sqrt{a^2+1}\right], \ \ \ \  &&0<\tau<\frac{\pi}{4},\\
&\left[\frac{s_2^2\sqrt{2}}{(1+s_2)^2}, n\sqrt{2}\right], &&\tau=\frac{\pi}{4},\\
&\left[\frac{s_2^2\sqrt{a^2+1}}{(s_2+a)^2+b^2}, n\sqrt{a^2+1}\right], &&\frac{\pi}{4}<\tau<\frac{\pi}{2},\\
&\left[\frac{s_2^2}{1+s_2^2}, n\right],  &&\tau=\frac{\pi}{2}.
\end{aligned} \right.
\end{equation}

For any $(\mu_1,\cdots,\mu_n)\in {\Gamma}^+_n$, denote
$$\lambda_i=\frac{1}{\mu_i},\ \ 1\leq i\leq n,$$
and
$$\tilde F_\tau (\mu_1,\cdots,\mu_n):=-F_\tau(\lambda_1,\cdots,\lambda_n).$$
Then
\begin{equation*}
\frac{\partial \tilde F_\tau}{\partial \mu_{i}}=\lambda^{2}_{i}\frac{\partial F_\tau}{\partial \lambda_{i}},\quad \mu^{2}_{i}\frac{\partial \tilde F_\tau}{\partial \mu_{i}}=\frac{\partial F_\tau}{\partial \lambda_{i}},
\end{equation*}
and
\begin{equation*}
\begin{aligned}
\frac{\partial^2 \tilde F_{\tau}}{\partial \mu_i \partial \mu_j}&&&=-\lambda^3_i\left(\lambda_i\frac{\partial^2 F_\tau}{\partial \lambda^2_i}+2\frac{\partial F_\tau}{\partial \lambda_i}\right)\delta_{ij}\\
&&&=\left\{ \begin{aligned}
& -\frac{2\sqrt{a^2+1}(\mu_{i}+a)}{\left[(1+a\mu_{i})^2-(b\mu_{i})^2\right]^2}\delta_{ij}, \ \ \ \  &&0<\tau<\frac{\pi}{4},\\
& -\frac{2\sqrt{2}\delta_{ij}}{(1+\mu_{i})^3}, &&\tau=\frac{\pi}{4},\\
& -\frac{2\sqrt{a^2+1}(\mu_i+a)}{\left[(1+a\mu_i)^2+(b\mu_{i})^2\right]^2}\delta_{ij}, &&\frac{\pi}{4}<\tau<\frac{\pi}{2},\\
& -\frac{2\mu_i\delta_{ij}}{\left(1+\mu^2_{i}\right)^2},  &&\tau=\frac{\pi}{2}.
\end{aligned} \right.
\end{aligned}
\end{equation*}
Therefore, we obtain
\begin{equation*}
\frac{\partial \tilde F_\tau}{\partial \mu_i}>0,\ \ 1\leq i\leq n\ \  \text{on}\ \  \Gamma^+_n,
\end{equation*}
and
\begin{equation}\label{e2.1.11}
\left(\frac{\partial^2 \tilde F_\tau}{\partial \mu_i\partial \mu_j}\right)\leq 0\ \  \text{on}\ \  \Gamma^+_n.
\end{equation}

By the discussion above, we have
\begin{Proposition}\label{prop2.10.2}
For $\tau\in (0,\frac{\pi}{2}]$,  the operator $F_\tau(\lambda)$,  satisfies the structure conditions (\ref{e1.2.0})-(\ref{e1.4}).
\end{Proposition}

In the next three sections, we are going to prove Theorem \ref{tttt1.1} through the short time existence of the parabolic flow, the strictly oblique estimate and the $C^2$ estimate based on a Schn$\ddot{\text{u}}$rer's convergence result.

\section{The short time existence of the parabolic flow}

Let $\mathscr{P}_n$ be the set of positive definite symmetric $n\times n$ matrices, and $\lambda_{1}(A)$, $\cdots$, $\lambda_{n}(A)$ be the eigenvalues of $A$.  For $A=(a_{ij})\in \mathscr{P}_n$, denote
$$F[A]:=F\left(\lambda_{1}(A),\cdots, \lambda_{n}(A)\right)$$
and
$$(a^{ij})=(a_{ij})^{-1},\,\,\, F^{ij}=\frac{\partial F}{\partial a_{ij}},\,\,\,F^{ij,rs}=\frac{\partial^{2} F}{\partial a_{ij}\partial a_{rs}}.$$
 Let us recall the relevant Sobolev spaces( cf. Chapter 1 in  \cite{W}). For every multi-index $\beta=(\beta_{1},\beta_{2},\cdots, \beta_{n})$, $\beta_{i}\geq 0$ for $i=1,2,\cdots,n$ with length $|\beta|=\sum^{n}_{i=1}\beta_{i}$ and $j\geq 0$, we set
$$D^{\beta}u:=\frac{\partial^{|\beta|}u}{\partial x_{1}^{\beta_{1}}\partial x_{2}^{\beta_{2}}\cdots\partial x_{n}^{\beta_{n}}},\quad
D^{\beta}D_{t}^{j}u:=\frac{\partial^{|\beta|+j}u}{\partial x_{1}^{\beta_{1}}\partial x_{2}^{\beta_{2}}\cdots\partial x_{n}^{\beta_{n}}\partial t^{j}}.$$
We state  the definition of the usual functional spaces as follows($k\geq 0$):
$$C^{k}(\Omega)=\{u:\Omega\rightarrow \mathbb{R}:\forall \beta,\ |\beta|\leq k,\ D^{\beta}u\  \text{is continuous in}\ \Omega\},$$
$$C^{k}(\bar{\Omega})=\{u\in C^{k}(\Omega) :\forall \beta, |\beta|\leq k, D^{\beta}u \text{ can be extended by continuity to } \partial\Omega\},$$
$$C^{k,\frac{k}{2}}(\Omega_{T})=\{u:\Omega_{T}\rightarrow \mathbb{R}:\forall \beta,j\geq 0, |\beta|+2j\leq k, D^{\beta}D^{j}_{t}u\  \text{is continuous in}\ \Omega_{T}\},$$
$$C^{k,\frac{k}{2}}(\bar{\Omega}_{T})=\{u\in C^{k,\frac{k}{2}}(\Omega_{T}) :\forall \beta,j\geq 0, |\beta|+2j\leq k, D^{\beta}D^{j}_{t}u\ \text{can be extended by continuity to}\ \partial\Omega_{T}\}.$$
Moreover $C^{k}(\bar{\Omega})$ and $C^{k,\frac{k}{2}}(\bar{\Omega}_{T})$ are Banach spaces equipped with the norm
$$\|u\|_{C^{k}(\bar{\Omega})}=\sum_{|\beta|\leq k}\sup_{\bar{\Omega}}|D^{\beta}u|$$
and
$$\|u\|_{C^{k,\frac{k}{2}}(\bar{\Omega}_{T})}=\sum_{j\geq 0,|\beta|+2j\leq k}\sup_{\bar{\Omega}_{T}}|D^{\beta}D^{j}_{t}u|$$
respectively.

We now present the definition of H\"{o}lder spaces. Let $\alpha\in [0,1]$, define the $\alpha$-H\"{o}lder coefficient of $u$ in $\Omega$ as
$$[u]_{\alpha, \Omega}=\sup_{x\neq y, x,y\in\Omega}\frac{|u(x)-u(y)|}{|x-y|^{\alpha}}.$$
If $[u]_{\alpha, \Omega}<+\infty,$ then we call $u$ H\"{o}lder continuous with exponent $\alpha$ in $\Omega.$
If there are not ambiguity about the domains $\Omega$, we denote $[u]_{\alpha, \Omega}$ by $[u]_{\alpha}$. Similarly, the $(\alpha,\frac{\alpha}{2})$-H\"{o}lder coefficient of $u$ in $\Omega_{T}$ can be defined by
$$[u]_{\alpha,\frac{\alpha}{2}, \Omega_{T}}=\sup_{(x,t)\neq (y,\tau), (x,t),(y,\tau)\in\Omega_{T}}\frac{|u(x,t)-u(y,\tau)|}{|x-y|^{\alpha}+|t-\tau|^{\frac{\alpha}{2}}},$$
and $u$ is H\"{o}lder continuous with exponent $(\alpha,\frac{\alpha}{2})$ in $\Omega_{T}$  if $[u]_{\alpha,\frac{\alpha}{2}, \Omega_{T}} <+\infty.$
Meanwhile, we denote $[u]_{\alpha,\frac{\alpha}{2}, \Omega_{T}}$ by  $[u]_{\alpha,\frac{\alpha}{2}}$.   We denote $C^{k+\alpha}(\bar{\Omega})$ as the set of functions belonging to
$C^{k}(\bar{\Omega})$ whose $k$-order partial derivatives are H\"{o}lder continuous with exponent $\alpha$ in $\Omega$ and $C^{k+\alpha}(\bar{\Omega})$ is
a Banach space equipped with the following norm
$$\|u\|_{C^{k+\alpha}(\bar{\Omega})}=\|u\|_{C^{k}(\bar{\Omega})}+[u]_{k+\alpha},$$
where
$$[u]_{k+\alpha}=\sum_{|\beta|=k}[D^{\beta}u]_{\alpha}.$$
Likewise, we denote $C^{k+\alpha,\frac{k+\alpha}{2}}(\bar{\Omega}_{T})$ as the set of functions belonging to $C^{k,\frac{k}{2}}(\bar{\Omega}_{T})$ whose $(k,\frac{k}{2})$-order partial derivatives are H\"{o}lder continuous with exponent $(\alpha,\frac{\alpha}{2})$ in $\Omega_{T}$ and $C^{k+\alpha,\frac{k+\alpha}{2}}(\bar{\Omega}_{T})$ is a Banach space equipped with the following norm
$$\|u\|_{C^{k+\alpha,\frac{k+\alpha}{2}}(\bar{\Omega}_{T})}=\|u\|_{C^{k,\frac{k}{2}}(\bar{\Omega}_{T})}+[u]_{k+\alpha, \frac{k+\alpha}{2}},$$
where
$$[u]_{k+\alpha,\frac{k+\alpha}{2}}=\sum_{|\beta|+2j=k}[D^{\beta}D^{j}_{t}u]_{\alpha,\frac{\alpha}{2}}.$$

By the methods on the second boundary value problems for equations of Monge-Amp\`{e}re type \cite{JU}, the parabolic boundary condition in (\ref{eeee1.2}) can be reformulated as
$$h(Du)=0,\qquad x\in \partial\Omega,\quad t>0,$$
where we need
\begin{deff}
A smooth function $h:\mathbb{R}^n\rightarrow\mathbb{R}$ is called the defining function of $\tilde{\Omega}$ if
$$\tilde{\Omega}=\{p\in\mathbb{R}^{n} : h(p)>0\},\quad |Dh|_{{\partial\tilde{\Omega}}}=1,$$
and there exists $\theta>0$ such that for any $p=(p_{1},\cdots, p_{n})\in \tilde{\Omega}$ and $\xi=(\xi_{1}, \cdots, \xi_{n})\in \mathbb{R}^{n}$,
$$\frac{\partial^{2}h}{\partial p_{i}\partial p_{j}}\xi_{i}\xi_{j}\leq -\theta|\xi|^{2}.$$
\end{deff}
We can also define $\tilde{h}$ as the defining function of $\Omega$. That is,
$$\Omega=\{\tilde{p}\in\mathbb{R}^{n} : \tilde{h}(\tilde{p})>0\},\ \ \ |D\tilde{h}|_{\partial\Omega}=1, \ \ \ D^2\tilde{h}\leq -\tilde{\theta}I,$$
where $\tilde{\theta}$ is some positive constant. Thus the parabolic flow (\ref{eeee1.1})-(\ref{eeee1.3}) is equivalent to the evolution problem
\begin{equation}\label{eeee2.1}
\left\{ \begin{aligned}\frac{\partial u}{\partial t}&=F\left(\lambda (D^2u)\right)-f(x),\ \ && t>0,\  x\in \Omega, \\
h(Du)&=0,&& t>0,\  x\in\partial\Omega,\\
 u&=u_{0}, &&  t=0,\  x\in \Omega.
\end{aligned} \right.
\end{equation}

To establish the short time existence of classical solutions of (\ref{eeee2.1}),  we use the inverse function theorem in Fr$\acute{e}$chet spaces and the theory of linear parabolic equations for oblique boundary condition. The method is along the idea of proving  the short time existence of convex solutions on the second boundary value problem for Lagrangian mean curvature flow \cite{HR}. We include the details for the convenience of the readers.

\begin{lemma}\label{l1.10}(I. Ekeland, see Theorem 2 in  \cite{IE}.)
Let $X$ and $Y$ be Banach spaces with the norms $\|\cdot\|_{1}$ and $\|\cdot\|_{2}$ respectively. Suppose $$\hbar: X\rightarrow Y$$
 is continuous and G\^{a}teaux-differentiable, with
$\hbar[0]=0$. Assume that the derivative $D\hbar[x]$ has a right inverse $\mathrm{T}[x]$, uniformly bounded in
a neighbourhood of $0$ in $X$. That is, for any $y\in Y$,
$$ D\hbar[x]\mathrm{T}[x]y=y ,$$
and there exist $R>0$ and $m>0$ such that
$$\|x\|_{1}\leq R\Longrightarrow \|\mathrm{T}[x]\|_{2}\leq m.$$
For every $y\in Y$, if
$$\| y\|_{2}<\frac{R}{m},$$
then there exists some $x\in X$ such that
$$\|x\|_{2}< R,$$
and
$$\hbar[x]=y.$$
\end{lemma}
As an application of Lemma \ref{l1.10}, we obtain the following inverse function theorem which will be used to prove the short time existence result for equation (\ref{eeee2.1}).

\begin{lemma}\label{l1.11}
(See Lemma 2.2 in \cite{HRY}.)
Let $X$ and $Y$ be Banach spaces with the norms $\|\cdot\|_{1}$ and $\|\cdot\|_{2}$ respectively. Suppose $$J: X\rightarrow Y$$
 is continuous and G\^{a}teaux-differentiable, with
$J(v_{0})=w_{0}$. Assume that the derivative $DJ[v]$ has a right inverse $L[v]$, uniformly bounded in
a neighbourhood of $v_{0}$. That is, for any $y\in Y$,
$$DJ[v]L[v]y=y,$$
and there exist $R>0$ and $m>0$ such that
$$\|v-v_{0}\|_{1}\leq R\Longrightarrow \|L[v]\|_{2}\leq m.$$
For every $w\in Y$, if
$$\| w-w_{0}\|_{2}<\frac{R}{m},$$
then there exists some $v\in X$ such that
$$\|v-v_{0}\|_{1}< R,$$
and
$$J(v)=w.$$
\end{lemma}

We will use the following short time existence and regularity resuts for linear second order  parabolic equation with strict oblique  boundary condition:
\begin{lemma}\label{l1.2}
(See Theorem 8.8 and 8.9 in \cite{GM}.) Assume that $\tilde f\in C^{\alpha_{0},\frac{\alpha_{0}}{2}}(\bar{\Omega}_{T})$ for some $0<\alpha_{0}<1$, $T>0$, and $G(x,p)$, $G_{p}(x,p)$ are in $C^{1+\alpha_{0}}(\Xi)$ for any compact subset $\Xi$ of $\partial\Omega\times\mathbb{R}^{n}$ such that $\inf_{\partial\Omega}\langle G_{p}, \nu\rangle>0$ where $\nu$ is the inner normal vector of $\partial\Omega$. Let $u_{0}\in C^{2+\alpha_{0}}(\bar{\Omega})$ be strictly convex and satisfy $G(x, Du_{0})=0.$ Then there exists $T'>0$ $(T'\leq T)$ such that we can find a unique solution which is strictly convex in $x$ variable in the class $C^{2+\alpha_{0},\frac{2+\alpha_{0}}{2}}(\bar{\Omega}_{T'})$ to the following equations
\begin{equation*}
\left\{ \begin{aligned}\frac{\partial u}{\partial t}-a^{ij}(x,t) u_{ij}&=\tilde f(x,t),\ \
&& T'>t>0,\  x\in \Omega, \\
G(x,Du)&=0,&& T'>t>0,\  x\in\partial\Omega,\\
 u&=u_{0}, && t=0,\  x\in \Omega,
\end{aligned} \right.
\end{equation*}
where  $a^{ij}(x,t)\in C^{\alpha_{0},\frac{\alpha_{0}}{2}}(\bar{\Omega}_{T})$, $1\leq i,j\leq n$ and $[a^{ij}(x,t)]\geq a_{0}\text{I}$ for some positive constant $a_{0}$.
\end{lemma}
By the property of $C^{2+\alpha_{0},\frac{2+\alpha_{0}}{2}}(\bar{\Omega}_{T'})$ and $u(x,t)|_{t=0}=u_{0}(x)$, we obtain
\begin{equation}\label{eeee2.00}
\lim_{t\rightarrow 0}\| u(\cdot,t)-u_{0}(\cdot)\|_{C^{2+\alpha_{0}}(\bar{\Omega})}=0.
\end{equation}
For any $\alpha<\alpha_{0}$, we have
\begin{equation*}
\begin{aligned}
&\frac{|(D^{2}u(x,t)-D^{2}u_{0}(x))-(D^{2}u(y,\tau)-D^{2}u_{0}(y))|}{|x-y|^{\alpha}+|t-\tau|^{\frac{\alpha}{2}}} \\
&\leq \frac{|(D^{2}u(x,t)-D^{2}u_{0}(x))-(D^{2}u(y,t)-D^{2}u_{0}(y))|}{|x-y|^{\alpha}}\\
&+|t-\tau|^{\frac{\alpha_{0}-\alpha}{2}}\frac{|(D^{2}u(y,t)-D^{2}u_{0}(y))-(D^{2}u(y,\tau)-D^{2}u_{0}(y))|}{|t-\tau|^{\frac{\alpha_{0}}{2}}}.
\end{aligned}
\end{equation*}
Then we get
\begin{equation}\label{eeee2.01}
\begin{aligned}
\| D^{2}u-D^{2}u_{0}\|_{C^{\alpha,\frac{\alpha}{2}}(\bar{\Omega}_{T'})}\leq &\max_{0\leq t\leq T'}\| D^{2}u(\cdot,t)-D^{2}u_{0}(\cdot)\|_{C^{\alpha}(\bar{\Omega})}\\
&+T'^{\frac{\alpha-\alpha_{0}}{2}}\| D^{2}u-D^{2}u_{0}\|_{C^{\alpha_{0},\frac{\alpha_{0}}{2}}(\bar{\Omega}_{T'})}.
\end{aligned}
\end{equation}
Combining  (\ref{eeee2.00}) with (\ref{eeee2.01}), we obtain
\begin{equation}\label{eeee2.02}
\lim_{T'\rightarrow 0}\| D^{2}u-D^{2}u_{0}\|_{C^{\alpha,\frac{\alpha}{2}}(\bar{\Omega}_{T'})}=0,
\end{equation}
which will be used later.

According to the proof in \cite{JU}, we can verify the oblique boundary condition.
\begin{lemma}\label{l1.3}(See J. Urbas \cite{JU}.) Let $\nu=(\nu_{1},\nu_{2}, \cdots,\nu_{n})$ be the unit inward normal vector of $\partial\Omega$. If $u\in C^{2}(\bar{\Omega})$  with $D^{2}u\geq0$, then there holds $h_{p_{k}}(Du)\nu_{k}\geq0$.
\end{lemma}

Now we can prove the short time existence of solutions of (\ref{eeee2.1}), which is equivalent to the problem (\ref{eeee1.1})-(\ref{eeee1.3}).
\begin{Proposition}\label{pppp1.1}
According to the conditions in Theorem \ref{tttt1.1}, there exist some $T''>0$ and $u\in C^{2+\alpha,\frac{2+\alpha}{2}}(\bar{\Omega}_{T''})$ which depend only on $\Omega$, $\tilde{\Omega}$, $u_0$, $f$, $\delta$ and $F$, such that $u$ is a solution of (\ref{eeee2.1}) and is strictly convex in $x$ variable.
\end{Proposition}

\begin{proof}
Denote the Banach spaces
$$X=C^{2+\alpha,1+\frac{\alpha}{2}}(\bar{\Omega}_{T}),
\quad Y=C^{\alpha,\frac{\alpha}{2}}(\bar{\Omega}_{T})
\times C^{1+\alpha,\frac{1+\alpha}{2}}(\partial\Omega\times(0,T])\times C^{2+\alpha}(\bar{\Omega}),$$
where
$$\|\cdot\|_{Y}=\|\cdot\|_{C^{\alpha,\frac{\alpha}{2}}(\bar{\Omega}_{T})}
+\|\cdot\|_{C^{1+\alpha,\frac{1+\alpha}{2}}(\partial\Omega\times(0,T])}+
\|\cdot\|_{C^{2+\alpha}(\bar{\Omega})}.$$
Define a map
$$J:\quad X\rightarrow Y$$
by
$$J(u)=\left\{ \begin{aligned}
&\frac{\partial u}{\partial t}-F[D^{2}u]+f(x), \ \ && (x,t)\in\Omega_{T}, \\
&h(Du), && (x,t)\in \partial\Omega\times(0,T],\\
&u, && (x,t)\in \Omega\times\{t=0\}.
\end{aligned} \right.$$
Thus the strategy is  to use the inverse function theorem to obtain the short time existence result.

The computation of the G\^{a}teaux derivative shows that for any $u$, $v\in X$,
$$DJ[u](v):=\frac{d}{d\tau}J(u+\tau v)|_{\tau=0}=\left\{ \begin{aligned}
&\frac{\partial v}{\partial t}-F^{ij}[D^{2}u]v_{ij},\  && (x,t)\in\Omega_{T}, \\
&h_{p_{i}}(Du)v_{i}, && (x,t)\in \partial\Omega\times(0,T],\\
&v, && (x,t)\in \Omega\times\{t=0\}.
\end{aligned} \right.$$
Using Lemma \ref{l1.2} and Lemma \ref{l1.3}, there exists $T_{1}>0$ such that  we can find $$\hat{u}\in C^{2+\alpha_{0},1+\frac{\alpha_{0}}{2}}(\bar{\Omega}_{T_{1}})\subset X$$ to be strictly convex in $x$ variable,  which satisfies the following equations
\begin{equation}\label{eeee2.2}
\left\{ \begin{aligned}\frac{\partial \hat{u}}{\partial t}-\Delta \hat{u}&=F[D^{2}u_{0}]-\Delta u_{0}-f,\ \
&& T_{1}>t>0,\ x\in \Omega, \\
h(D\hat{u})&=0,&& T_{1}>t>0,\  x\in\partial\Omega,\\
 \hat{u}&=u_{0}, && t=0,\ x\in \Omega.
\end{aligned} \right.
\end{equation}
We see that there exists $ R>0$, such that $u$ is strictly convex in $x$ variable if $$\|u-\hat{u}\|_{C^{2+\alpha,\frac{2+\alpha}{2}}(\bar{\Omega}_{T_{1}})}<R.$$

For each $Z:=(\bar f,\bar g,\bar w)\in Y$, using Lemma \ref{l1.2} again, we know that there exists a unique $v\in X\ (T=T_{1})$ satisfying $DJ[u](v)=(\bar f,\bar g,\bar w)$, that is,
\begin{equation*}
\left\{ \begin{aligned}
\frac{\partial v}{\partial t}-F^{ij}[D^{2}u]v_{ij}&=\bar f,\ \ && T_{1}>t>0,\ x\in \Omega, \\
h_{p_{i}}(Du)v_{i}&=\bar g,&& T_{1}>t>0,\ x\in\partial\Omega,\\
 v&=\bar w, &&  t=0,\ x\in \Omega.
\end{aligned} \right.
\end{equation*}
Using Schauder estimates for linear parabolic equation to oblique boundary condition(cf. Theorem 8.8 and 8.9 in \cite{GM}), we obtain for some positive constant $m$,
\begin{equation*}
\| v \|_{C^{2+\alpha,\frac{2+\alpha}{2}}(\bar{\Omega}_{T_1})}
\leq
m\left(\|\bar f \|_{C^{\alpha,\frac{\alpha}{2}}(\bar{\Omega}_{T_1})}
+\|\bar g \|_{C^{1+\alpha,\frac{1+\alpha}{2}}(\partial\Omega\times(0,T_1])}+
\|\bar w \|_{C^{2+\alpha}(\bar{\Omega})}\right).
\end{equation*}
For $T=T_1$, by the definition of the Banach spaces $X$ and $Y$, we can rewrite the above Schauder estimates as
\begin{equation*}
\| v \|_{X} \leq m \| Z \|_{Y}.
\end{equation*}
If  $\|Z\|_{Y}\leq 1$, then we have
$$\|v\|_{X}\leq m. $$
It means that the derivative $DJ[u](v)=Z$ has a right inverse $v=L[u](Z)$ and
\begin{equation*}
\|L[u]\|:=\sup_{\|Z\|_{Y}\leq 1}\|L[u](Z)\|_{X}\leq m.
\end{equation*}
If  we set
$$\hat{f}=\frac{\partial \hat{u}}{\partial t}-F[D^{2}\hat{u}]+f,\ \ w_{0}=(\hat{f}, 0,u_{0}),\ \ w=(0, 0,u_{0}),$$
then we can show that
\begin{equation*}
\begin{aligned}
\|\hat{f}-0\|_{C^{\alpha,\frac{\alpha}{2}}(\bar{\Omega}_{T_{1}})}
&=\|\Delta \hat{u}-\Delta u_{0}+F[D^{2}u_{0}]-F[D^{2}\hat{u}]\|_{C^{\alpha,\frac{\alpha}{2}}(\bar{\Omega}_{T_{1}})}\\
&\leq\|\Delta \hat{u}-\Delta u_{0}\|_{C^{\alpha,\frac{\alpha}{2}}(\bar{\Omega}_{T_{1}})}+\| F[D^{2}u_{0}]-F[D^{2}\hat{u}]\|_{C^{\alpha,\frac{\alpha}{2}}(\bar{\Omega}_{T_{1}})}\\
&\leq C\| D^{2}\hat{u}-D^{2}u_{0}\|_{C^{\alpha,\frac{\alpha}{2}}(\bar{\Omega}_{T_{1}})},
\end{aligned}
\end{equation*}
where $C$ is a constant depending only on the known data. Using (\ref{eeee2.02}), we conclude that there exists $T''>0$ $(T''\leq T_{1})$ to be small enough such that
$$\|\hat{f}-0\|_{C^{\alpha,\frac{\alpha}{2}}(\bar{\Omega}_{T''})}\leq  C\| D^{2}\hat{u}-D^{2}u_{0}\|_{C^{\alpha,\frac{\alpha}{2}}(\bar{\Omega}_{T''})}<\frac{R}{m}.$$
Therefore,
$$\| w-w_{0}\|_{Y}=
\|0-\hat{f}\|_{C^{\alpha,\frac{\alpha}{2}}(\bar{\Omega}_{T''})}<\frac{R}{m}.
$$
By Lemma \ref{l1.11}, we obtain the desired result.
\end{proof}
\begin{rem}
By the strong maximum principle, the strictly convex solution to (\ref{eeee2.1}) is unique.
\end{rem}

\section{The strict obliqueness estimate and the $C^2$ estimate}
In this section, the $C^{2}$ a priori bound is accomplished by making the second derivative estimates on the boundary for the solutions of fully nonlinear parabolic equations. We also refer to the recent preprint \cite{WHB} for a proof of separation in elliptic setting with the same criterion as the one used in the present work. This treatment is similar to the problems presented in \cite{HR}, \cite{JU} and \cite{OK}, but requires some modification to accommodate the more general situation. Specifically, the structure conditions (\ref{e1.3}) and (\ref{e1.4}) are needed in order to derive differential inequalities from barriers which can be used.

For the convenience, we denote $\beta=(\beta^{1}, \cdots, \beta^{n})$ with $\beta^{i}:=h_{p_{i}}(Du)$, and $\nu=(\nu_{1},\cdots,\nu_{n})$ as the unit inward normal vector at $x\in\partial\Omega$. The expression of the inner product is
\begin{equation*}
\langle\beta, \nu\rangle=\beta^{i}\nu_{i}.
\end{equation*}
By Proposition \ref{pppp1.1} and the regularity theory of parabolic equations, we may assume that $u$ is a strictly convex solution of (\ref{eeee1.1})-(\ref{eeee1.3}) in the class $C^{2+\alpha,1+\frac{\alpha}{2}}(\bar{\Omega}_{T})\cap C^{4+\alpha,2+\frac{\alpha}{2}}(\Omega_{T})$ for some $T>0$.
\begin{lemma}\label{11115.1a}($\dot{u}$-estimates) If the convex solution to (\ref{eeee1.1})-(\ref{eeee1.3}) exists and $f\in \mathscr{A}_\delta$, then
\begin{equation*}
\min_{\bar{\Omega}}F[D^{2}u_{0}]-\max_{\bar{\Omega}}f(x)\leq\dot{u}\leq\max_{\bar{\Omega}}F[D^{2}u_{0}]-\min_{\bar{\Omega}}f(x),
\end{equation*}
where $\dot{u}:=\frac{\partial u}{\partial t}$.
\end{lemma}
\begin{proof} From (\ref{eeee1.1}), a direct computation shows that
$$\frac{\partial(\dot{u}) }{\partial t}-F^{ij}\partial_{ij}(\dot{u})=0.$$
Using the maximum principle, we see that
$$\min_{\bar{\Omega}_{T}}(\dot{u}) =\min_{\partial\bar{\Omega}_{T}}(\dot{u}).$$
Without loss of generality, we assume that $\dot{u}\neq constant$. If there exists $x_{0}\in \partial\Omega$, $t_{0}>0$, such that $\dot{u}(x_{0},t_{0})=\min_{\bar{\Omega}_{T}}(\dot{u})$. On the one hand, since $\langle\beta, \nu\rangle>0$, by the Hopf Lemma (cf.\cite{LL}) for parabolic equations, there must hold in the following
$$\dot{u}_{\beta}(x_{0},t_{0})\neq 0.$$
On the other hand,  we differentiate the boundary condition and then obtain
$$\dot{u}_{\beta}=h_{p_{k}}(Du)\frac{\partial \dot{u}}{\partial x_{k}}=\frac{\partial h(Du)}{\partial t}=0.$$
It is a contradiction. So we deduce that
$$\dot{u}\geq \min_{\bar{\Omega}_{T}}(\dot{u})
=\min_{\partial\bar{\Omega}_{T}|_{t=0}}(\dot{u})=\min_{\bar{\Omega}}\left(F[D^{2}u_{0}]-f(x)\right)\geq \min_{\bar{\Omega}}F[D^{2}u_{0}]-\max_{\bar{\Omega}}f(x).$$
For the same reason, we have
$$\dot{u}\leq \max_{\bar{\Omega}_{T}}(\dot{u})
=\max_{\partial\bar{\Omega}_{T}|_{t=0}}(\dot{u})=\max_{\bar{\Omega}}\left(F[D^{2}u_{0}]-f(x)\right)\leq \max_{\bar{\Omega}}F[D^{2}u_{0}]-\min_{\bar{\Omega}}f(x).$$
Putting these facts together, the assertion follows.
\end{proof}

\begin{lemma}\label{llll5.1b}
Let $(x,t)$ be an arbitrary point of $\Omega_{T}$, and $\lambda_{1}(x,t)$, $\cdots$, $\lambda_{n}(x,t)$ be the eigenvalues of $D^{2}u$ at $(x,t)$. Suppose that (\ref{e1.2.0}) and (\ref{e1.2}) hold, if $\mathop{\operatorname{osc}}_{\bar{\Omega}}(f) \leq\delta$ and $u$ is a strictly convex solution to (\ref{eeee1.1})-(\ref{eeee1.3}), then there exists $\mu>0$ and $\omega>0$ depending only on $F[D^{2}u_{0}]$ and $\delta$ such that
\begin{equation*}
\min_{1\leq i\leq n}\lambda_i(x,t)\leq\mu,\ \  \max_{1\leq i\leq n}\lambda_i(x,t)\geq\omega.
\end{equation*}
\end{lemma}
\begin{proof}
By condition (\ref{e1.2}) and Lemma \ref{11115.1a}, we obtain
\begin{equation*}
\begin{aligned}
F\left(\min_{1\leq i\leq n}\lambda_i(x,t),\cdots, \min_{1\leq i\leq n}\lambda_i(x,t)\right)&\leq F[D^2u]=\dot{u}+f(x)\\
&\leq\max_{\bar{\Omega}}F[D^{2}u_{0}]+f(x)-\min_{\bar{\Omega}}f(x)\\
&\leq \max_{\bar{\Omega}}F[D^{2}u_{0}]+\mathop{\operatorname{osc}}_{\bar{\Omega}}(f)\\
&\leq \max_{\bar{\Omega}}F[D^{2}u_{0}]+\delta\\
&< F\left(+\infty,\cdots,+\infty\right),
\end{aligned}
\end{equation*}
and
\begin{equation*}
\begin{aligned}
F\left(\max_{1\leq i\leq n}\lambda_i(x,t),\cdots, \max_{1\leq i\leq n}\lambda_i(x,t)\right)&\geq F[D^2u]=\dot{u}+f(x)\\
&\geq\min_{\bar{\Omega}}F[D^{2}u_{0}]+f(x)-\max_{\bar{\Omega}}f(x)\\
&\geq \min_{\bar{\Omega}}F[D^{2}u_{0}]-\mathop{\operatorname{osc}}_{\bar{\Omega}}(f)\\
&\geq \min_{\bar{\Omega}}F[D^{2}u_{0}]-\delta\\
&> F(0,\cdots,0).
\end{aligned}
\end{equation*}
By the monotonicity of $F$ and condition (\ref{e1.2.0}), we get the desired result.
\end{proof}

By Lemma \ref{llll5.1b}, the points $(\lambda_{1},\lambda_{2},\cdots, \lambda_{n})$ are always in $ \Gamma^{+}_{]\mu,\omega[}$ under the flow. So we can obtain
\begin{lemma}\label{llll5.1c}
Let $(x,t)$ be an arbitrary point of $\Omega_{T}$, and $\lambda_{1}(x,t)$, $\cdots$, $\lambda_{n}(x,t)$ be the eigenvalues of $D^{2}u$ at $(x,t)$. Suppose that (\ref{e1.2.0}) and (\ref{e1.2}) hold, if $\mathop{\operatorname{osc}}_{\bar{\Omega}}(f) \leq\delta$ and $u$ is a strictly convex solution to (\ref{eeee1.1})-(\ref{eeee1.3}), then there exists $\Lambda_1>0$ and $\Lambda_2>0$ depending only on $F[D^{2}u_{0}]$ and $\delta$ such that $F$ satisfies the structure conditions (\ref{e1.3}) and (\ref{e1.4}).
\end{lemma}
In the following, we always assume that $\Lambda_1>0$ and $\Lambda_2>0$ are universal constants depending on the known data.

For technical needs below, we introduce the Legendre transformation of $u$. For any $x\in \mathbb{R}^n$, define
\begin{equation*}
\tilde{x}_{i}:=\frac{\partial u}{\partial x_{i}}(x),\ \ i=1,2,\cdots,n,
\end{equation*}
and
$$\tilde u(\tilde{x}_{1},\cdots,\tilde{x}_{n},t):=\sum_{i=1}^{n}x_{i}\frac{\partial u}{\partial x_{i}}(x,t)-u(x,t).$$
In terms of $\tilde{x}_{1}$, $\cdots$, $\tilde{x}_{n}$ and $\tilde u(\tilde{x}_{1},\cdots,\tilde{x}_{n},t)$, we can easily check that
$$\left(\frac{\partial^{2} \tilde{u}}{\partial \tilde{x}_{i}\partial
\tilde{x}_{j}}\right)=\left(\frac{\partial^{2} u}{\partial x_{i}
\partial x_{j}}\right)^{-1}.$$
Let $\mu_{1}$, $\cdots$, $\mu_{n}$ be the eigenvalues of $D^{2}\tilde{u}$ at $\tilde{x}=D u(x)$. We denote
$$\mu_{i}=\lambda^{-1}_{i},\ \ i=1,2,\cdots,n.$$
Then
$$\frac{\partial \tilde F}{\partial \mu_{i}}=\lambda^{2}_{i}\frac{\partial F}{\partial \lambda_{i}},\quad \mu^{2}_{i}\frac{\partial \tilde F}{\partial \mu_{i}}=\frac{\partial F}{\partial \lambda_{i}}.$$
Moreover, it follows from (\ref{eeee2.1}) that
\begin{equation}\label{eeee2.1.1001}
\left\{ \begin{aligned}
\frac{\partial \tilde{u}}{\partial t}&=\tilde{F}(D^{2}\tilde{u})+f(D\tilde{u}),\ \ && t>0,\ \tilde x\in \tilde{\Omega}, \\
\tilde{h}(D\tilde{u})&=0,&& t>0,\ \tilde x\in\partial\tilde{\Omega},\\
 \tilde{u}&=\tilde{u}_{0}, &&  t=0,\ \tilde x\in \tilde{\Omega},
\end{aligned} \right.
\end{equation}
where $\tilde h$ is the defining function of $\Omega$, and $\tilde{u}_{0}$ is the Legendre transformation of $u_{0}$.

\begin{rem}\label{r2.2}
By Lemma \ref{llll5.1b}, if $u$ is a strictly convex solution to (\ref{eeee1.1})-(\ref{eeee1.3}), then the eigenvalues of $D^{2}u$ and $D^{2}\tilde{u}$ must be in $\Gamma^{+}_{]\mu,\omega[}$ and $\Gamma^{+}_{]\omega^{-1},\mu^{-1}[}$ respectively. Therefore, $\tilde F$ also satisfies the structure conditions (\ref{e1.3}) and (\ref{e1.4}).
\end{rem}

In order to establish the $C^{2}$ estimates,  we make use of the method to do the strict obliqueness estimates, a parabolic version of a result of J.Urbas \cite{JU} which was given in \cite{OK}. Returning to Lemma \ref{l1.3}, we get a uniform positive lower bound of the quantity $\inf_{\partial\Omega}h_{p_{k}}(Du)\nu_{k}$ which does not depend on $t$ under the structure conditions of $F$.

\begin{lemma}\label{llll3.4}
Let $F$ satisfy the structure conditions (\ref{e1.2.0})-(\ref{e1.4}) and $f\in {\mathscr{A}}_\delta$. If $u$ is a strictly convex solution to (\ref{eeee1.1})-(\ref{eeee1.3}) and $|D f|$ is sufficiently small, then the strict obliqueness estimate
\begin{equation}\label{eeee3.4}
\langle\beta, \nu\rangle\geq \frac{1}{C_1}>0
\end{equation}
holds on $\partial \Omega$ for some universal constant $C_1$, which depends only on $F$, $u_0$, $\Omega,$ $\tilde{\Omega}$ and $\delta$, and is independent of $t$.
\end{lemma}

\begin{proof}
The proof  follows the similar computations carried out in \cite{WHB}.

Define
$$v=\langle\beta, \nu\rangle+h(Du).$$
Let $(x_{0},t_{0})\in \partial\Omega\times[0,T]$ such that
$$\langle\beta, \nu\rangle(x_{0},t_{0})=h_{p_{k}}(Du(x_{0},t_{0}))\nu_{k}(x_{0},t_{0})=\min_{\partial\Omega\times[0,T]}\langle\beta, \nu\rangle.$$
By rotation, we may assume that $t_{0}>0$ and $\nu(x_{0},t_0)=(0,0,\cdots,1)=:e_{n}$. By the above assumptions and the boundary condition, we obtain
$$v(x_{0},t_{0})=\min_{\partial\Omega\times[0,T]}v=\min_{\partial\Omega\times[0,T]}\langle\beta, \nu\rangle=h_{p_{n}}(Du(x_{0},t_{0})).$$
By the convexity of $\Omega$ and its smoothness, we extend $\nu$ smoothly to a tubular neighborhood of $\partial\Omega$ such that in the matrix sense
\begin{equation}\label{eeee3.5}
\left(\nu_{kl}\right):=\left(D_k\nu_l\right)\leq -\frac{1}{C}\operatorname{diag} (1,\cdots, 1,0),
\end{equation}
where $C$ is a positive constant. By Lemma \ref{l1.3}, we see that $h_{p_{n}}(Du(x_{0},t_0))\geq0$.

At $(x_{0},t_{0})$ we have
\begin{equation}\label{eeee3.6}
0=v_{r}=h_{p_{n}p_{k}}u_{kr}+ h_{p_{k}}\nu_{kr}+h_{p_{k}}u_{kr},\quad 1\leq r\leq n-1.
\end{equation}
At this point we point out a key  estimate
\begin{equation}\label{eeee3.8}
v_{n}(x_{0},t_{0})\geq -C
\end{equation}
 which will be proved later, where $C$ is a constant depending only on $\Omega$, $u_{0}$, $h$, $\tilde{h}$ and $\delta$.

It's not hard to check that (\ref{eeee3.8}) can be rewritten as
\begin{equation}\label{eeee3.9}
h_{p_{n}p_{k}}u_{kn}+ h_{p_{k}}\nu_{kn}+h_{p_{k}}u_{kn}\geq -C.
\end{equation}
Multiplying (\ref{eeee3.9}) with $h_{p_{n}}$ and (\ref{eeee3.6}) with $h_{p_{r}}$ respectively, and summing up together, we obtain
\begin{equation}\label{eeee3.10}
h_{p_{k}}h_{p_{l}}u_{kl}\geq -Ch_{p_{n}}-h_{p_{k}}h_{p_{l}}\nu_{kl}-h_{p_{k}}h_{p_{n}p_{l}}u_{kl}.
\end{equation}
Using (\ref{eeee3.5}), and
$$ 1\leq r\leq n-1,\quad h_{p_k}u_{kr}=\frac{\partial h(Du)}{\partial x_r}=0,\quad h_{p_k}u_{kn}=\frac{\partial h(Du)}{\partial x_n}\geq 0,\quad -h_{p_np_n}\geq 0,$$
we have
$$h_{p_k}h_{p_l}u_{kl}\geq-Ch_{p_n}+\frac{1}{C}|Dh|^2-\frac{1}{C}h^2_{p_n}=-Ch_{p_n}+\frac{1}{C}-\frac{1}{C}h^2_{p_n}.$$
For the last term of the above inequality, we distinguish two cases at $(x_{0},t_{0})$.

Case (i). If
$$-Ch_{p_n}+\frac{1}{C}-\frac{1}{C}h^2_{p_n}\leq \frac{1}{2C},$$
then
$$h_{p_k}(Du)\nu_{k}=h_{p_n }\geq \sqrt{\frac{1}{2}+\frac{C^4}{4}}-\frac{C^2}{2}.$$
It shows that there is a uniform positive lower bound for the quantity $\min_{\partial\Omega\times[0,T]}h_{p_{k}}(Du)\nu_{k}$.

Case (ii). If
$$-Ch_{p_n}+\frac{1}{C}-\frac{1}{C}h^2_{p_n}> \frac{1}{2C},$$
then we obtain a positive lower bound of $h_{p_k}h_{p_l}u_{kl}$.

Let $\tilde{u}$ be the Legendre transformation of $u$, then $\tilde{u}$ satisfies
\begin{equation}\label{eeee3.12}
\left\{ \begin{aligned}
\frac{\partial \tilde{u}}{\partial t}&=\tilde{F}(D^{2}\tilde{u})+f(D\tilde{u}),\ \ && T>t>0,\ \tilde{x}\in \tilde{\Omega}, \\
\tilde{h}(D\tilde{u})&=0,&& T>t>0,\ \tilde{x}\in\partial\tilde{\Omega},\\
 \tilde{u}&=\tilde{u}_{0}, &&  t=0,\ \tilde{x}\in \tilde{\Omega},
\end{aligned} \right.
\end{equation}
where $\tilde h$ is the defining function of $\Omega$, and $\tilde{u}_{0}$ is the Legendre transformation of $u_{0}$. The unit inward normal vector of $\partial\Omega$ can be expressed by $\nu=D\tilde{h}$. For the same reason, $\tilde{\nu}=Dh$, where $\tilde{\nu}=(\tilde{\nu}_{1}, \tilde{\nu}_{2},\cdots,\tilde{\nu}_{n})$ is the unit inward normal vector of $\partial\tilde{\Omega}$.

Let $\tilde{\beta}=(\tilde{\beta}^{1}, \cdots, \tilde{\beta}^{n})$ with $\tilde{\beta}^{k}:=\tilde{h}_{p_{k}}(D\tilde{u})$.
We note that one can also define
$$\tilde{v}=\langle\tilde{\beta}, \tilde{\nu}\rangle+\tilde{h}(D\tilde{u}),$$
in which
$$\langle\tilde{\beta}, \tilde{\nu}\rangle=\langle\beta, \nu\rangle.$$
Denote $\tilde{x}_{0}=Du(x_{0})$. Then we obtain
$$\tilde{v}(\tilde{x}_{0},t_0)=v(x_{0},t_0)=\min_{\partial\tilde{\Omega}\times [0,T]} \tilde{v}.$$
Using the same methods, under the assumption of
\begin{equation}\label{eeee3.13}
\tilde{v}_{n}(\tilde{x}_{0},t_{0})\geq -C,
\end{equation}
we obtain the positive lower bounds of $\tilde{h}_{p_{k}}\tilde{h}_{p_{l}}\tilde{u}_{kl}$, or
$$h_{p_{k}}(Du)\nu_{k}=\tilde{h}_{p_{k}}(D\tilde{u})\tilde{\nu}_{k}=\tilde{h}_{p_{n}}\geq\sqrt{\frac{1}{2}+\frac{C^4}{4}}-\frac{C^2}{2}.$$
We notice that
$$\tilde{h}_{p_{k}}\tilde{h}_{p_{l}}\tilde{u}_{kl}=\nu_{i}\nu_{j}u^{ij}.$$
Then by the positive lower bounds of $h_{p_{k}}h_{p_{l}}u_{kl}$ and $\tilde{h}_{p_{k}}\tilde{h}_{p_{l}}\tilde{u}_{kl}$, the desired result follows from
\begin{equation}\label{eeee3.0}
\langle\beta, \nu\rangle=\sqrt{u^{ij}\nu_{i}\nu_{j}h_{p_{k}}h_{p_{l}}u_{kl}},
\end{equation}
which is proved in \cite{JU}.

It remains to prove the key estimate (\ref{eeee3.8}) and (\ref{eeee3.13}).

We prove (\ref{eeee3.8}) first. By $D^2\tilde{h}\leq -\tilde{\theta}I$ and (\ref{e1.3}) we have
\begin{equation}\label{eeee3.151}
L\tilde{h}\leq -\tilde{\theta}\sum^{n}_{i=1} F^{ii},
\end{equation}
where
$$L:=F^{ij}\partial_{ij}-\partial_{t}.$$
On the other hand,
\begin{equation}\label{eeee3.100}
 \begin{aligned}
Lv=&h_{p_kp_lp_m}\nu_kF^{ij}u_{li}u_{mj}+2h_{p_kp_l}F^{ij}\nu_{kj}u_{li}\\
&+h_{p_kp_l}F^{ij}u_{lj}u_{ki}+h_{p_kp_l}\nu_kLu_l+h_{p_k}L\nu_k+h_{p_k}Lu_k.
 \end{aligned}
\end{equation}
Now we estimate the first term on the right hand side of (\ref{eeee3.100}).  By the diagonal basis and (\ref{e1.4}), we have
$$|h_{p_kp_lp_m}\nu_kF^{ij}u_{li}u_{mj}|\leq C\sum_{i=1}^{n}\frac{\partial F}{\partial \lambda_i}\lambda_i^2\leq C, $$
where $C$ is a constant depending only on $h$, $\Omega$, $\Lambda_1$, $\Lambda_2$, $u_0$ and $\delta$. Similarly, we also get
$$ |h_{p_kp_l}F^{ij}u_{lj}u_{ki}|\leq C\sum_{i=1}^{n}\frac{\partial F}{\partial \lambda_i}\lambda_i^2\leq C.$$
For the second term, by Cauchy inequality, we obtain
\begin{equation*}
\begin{aligned}
|2h_{p_{k}p_{l}}F^{ij}\nu_{kj}u_{li}|&\leq C\sum^{n}_{i=1}\frac{\partial F}{\partial \lambda_{i}}\lambda_{i}
=C\sum^{n}_{i=1}\sqrt{\frac{\partial F}{\partial \lambda_{i}}}\sqrt{\frac{\partial F}{\partial \lambda_{i}}}\lambda_{i}\\
&\leq C\left(\sum^{n}_{i=1}\frac{\partial F}{\partial \lambda_{i}}\right)\left(\sum^{n}_{i=1}\frac{\partial F}{\partial \lambda_{i}}\lambda^{2}_{i}\right)\\
&\leq C.
\end{aligned}
\end{equation*}
By (\ref{eeee1.1}) we have $Lu_l=f_{l}$. Then we get
$$|h_{p_kp_l}\nu_kLu_l|\leq C,\quad |h_{p_k}Lu_k|\leq C.$$
It follows from (\ref{e1.3}) that
$$|h_{p_{k}}L\nu_k|\leq C\sum^{n}_{i=1} F^{ii}.$$
Inserting these into (\ref{eeee3.100}) and using (\ref{e1.3}), it is immediate to check that there exists a positive constant $C$ depending only on $h$, $\Omega$, $\Lambda_1$, $\Lambda_2$, $u_0$ and $\delta$, such that
\begin{equation}\label{eeee3.15}
Lv\leq C\sum^{n}_{i=1} F^{ii}.
\end{equation}
Denote a neighborhood of $x_0$ in $\Omega$ by
$$\Omega_{\rho}:=\Omega\cap B_{\rho}(x_0),$$
where $\rho$ is a positive constant such that $\nu$ is well defined in $\Omega_{\rho}$.
To obtain the key estimate, we need to  consider the function
$$\Phi(x):=v(x,t)-v(x_{0},t_{0})+C_0\tilde{h}(x)+A|x-x_0|^2,$$
where $C_0$ and $A$ are positive constants to be determined. On $\partial\Omega\times[0,T]$, it is clear that $\Phi\geq 0$. Since $v$ is bounded, we can choose $A$ large enough such that on $\left(\Omega\cap \partial B_{\rho}(x_0)\right)\times[0,T]$
$$\Phi(x)=v(x,t)-v(x_{0},t_{0})+C_{0}\tilde{h}(x)+A|x-x_{0}|^{2}\geq v(x,t)-v(x_{0},t_{0})+A\rho^{2}\geq 0.$$
By the strict concavity of $\tilde{h}$, we have
$$ \Delta(C_{0}\tilde{h}(x)+A|x-x_{0}|^{2})\leq C(-C_{0}\tilde{\theta}+2A)\sum_{i=1}^n F^{ii}.$$
Then by choosing $C_{0}\gg A$, we obtain
$$\Delta(v(x,0)-v(x_{0},t_{0})+C_{0}\tilde{h}(x)+A|x-x_{0}|^{2})\leq 0.$$
We apply the maximum principle to get
\begin{equation*}
\begin{aligned}
&(v(x,0)-v(x_{0},t_{0})+C_{0}\tilde{h}(x)+A|x-x_{0}|^{2})|_{\Omega_{\rho}}\\
&\geq\min_{(\partial\Omega\cap B_{\rho}(x_{0}))\cup(\Omega\cap\partial B_{\rho}(x_{0}))} (v(x,0)-v(x_{0},t_{0})+C_{0}\tilde{h}(x)+A|x-x_{0}|^{2})\\
&\geq 0.
\end{aligned}
\end{equation*}
Combining (\ref{eeee3.151}) with (\ref{eeee3.15}) and letting $C_{0}$ be large enough, one yields
$$L\Phi\leq (-C_{0}\tilde{\theta}+C+2A)\sum_{i=1}^n F^{ii}\leq 0.$$
From the above arguments, we verify that $\Phi$ satisfies
\begin{equation}\label{eeee3.16}
\left\{ \begin{aligned}L\Phi&\leq 0,\ \ &&(x,t)\in\Omega_{\rho}\times[0,T] , \\
\Phi&\geq 0, &&(x,t)\in(\partial\Omega_{\rho}\times[0,T])\cup(\Omega_{\rho}\times\{t=0\}).
\end{aligned} \right.
\end{equation}
Using the maximum principle, we deduce that
$$\Phi\geq 0,\quad (x,t)\in\Omega_{\rho}\times[0,T].$$
Combining it with $\Phi(x_{0},t_{0})=0$, we obtain $\langle\nabla\Phi,e_{n}\rangle|_{(x_{0},t_{0})}\geq0$, which gives the desired key estimate (\ref{eeee3.8}).

Finally, we prove (\ref{eeee3.13}). The proof of (\ref{eeee3.13}) is similar to the one of (\ref{eeee3.8}). Define
$$\tilde L=\tilde{F}^{ij}\partial_{ij}+f_{p_i}\partial_i-\partial_t.$$
By (\ref{eeee3.12}) we see that $\tilde L\tilde u_l=0$, and thus
$$\tilde L\tilde v=\tilde F^{ij}\tilde u_{mj}\tilde u_{li}\tilde h_{p_kp_lp_m}\tilde \nu_k+2\tilde h_{p_kp_l}\tilde F^{ij}\tilde u_{li}\tilde \nu_{kj}+\tilde F^{ij}\tilde h_{p_k}\tilde \nu_{kij}+\tilde h_{p_kp_l}\tilde F^{ij}\tilde u_{lj}\tilde u_{ki}+\tilde h_{p_k}f_{p_i}\tilde \nu_{ki}.$$
By making use of  the following identities
\begin{equation*}
\frac{\partial \tilde F}{\partial \mu_{i}}=\lambda^{2}_{i}\frac{\partial F}{\partial \lambda_{i}},\quad \mu^{2}_{i}\frac{\partial \tilde F}{\partial \mu_{i}}=\frac{\partial F}{\partial \lambda_{i}}.
\end{equation*}
we deduce that $\tilde{F}$ satisfies the structure conditions (\ref{e1.2.0})-(\ref{e1.4}). Repeating the proof of (\ref{eeee3.15}), we have
\begin{equation}\label{eqq2.13aa}
\tilde L\tilde v\leq C\sum^{n}_{i=1} \tilde F^{ii},
\end{equation}
where $C$ depends only on $\tilde{\Omega}$, $\Omega$, $\Lambda_1$, $\Lambda_2$, $\delta$ and $u_0$.

Denote a neighborhood of $\tilde{x}_0$ in $\tilde\Omega$ by
$$\tilde\Omega_{r}:=\tilde\Omega\cap B_{r}(\tilde{x}_0),$$
where $r$ is a positive constant such that $\tilde\nu$ is well defined in $\tilde\Omega_{r}$. Consider
$$\tilde\Phi(y):=\tilde v(y,t)-\tilde v(\tilde{x}_0,t_0)+\tilde C_0 h(y)+\tilde A|y-\tilde{x}_0|^2,$$
where $\tilde C_0$ and $\tilde A$ are positive constants to be determined. It is clear that $\tilde\Phi\geq 0$ on $\partial\tilde \Omega\times [0,T]$. Since $\tilde v$ is bounded, we can choose $\tilde A$ large enough such that on $\left(\tilde\Omega\cap \partial B_{r}(\tilde{x}_0)\right)\times[0,T]$
$$\tilde\Phi(y)\geq\tilde v(y,t)-\tilde v(\tilde{x}_0,t_0)+\tilde A r^2\geq 0.$$
By the strict concavity of $h$, we have
$$ \Delta\left(\tilde C_{0}h(y)+\tilde A|y-\tilde{x}_0|^{2}\right)\leq C(-\tilde C_{0}\theta+2\tilde A)\sum_{i=1}^n\tilde F^{ii}.$$
Then by choosing $\tilde C_{0}\gg \tilde A$, we have
$$\Delta\left(\tilde v(y,0)-\tilde v(\tilde{x}_0,t_{0})+\tilde C_{0}h(y)+\tilde A|y-\tilde{x}_0|^{2}\right)\leq 0.$$
It follows from the maximum principle that
\begin{equation*}
\begin{aligned}
&(\tilde v(y,0)-\tilde v(\tilde{x}_0,t_{0})+\tilde C_{0}h(y)+\tilde A|y-\tilde{x}_0|^{2})|_{\tilde \Omega_{r}}\\
&\geq\min_{(\partial\tilde\Omega\cap B_{r}(\tilde{x}_0))\cup(\tilde\Omega\cap\partial B_{r}(\tilde{x}_0))} (\tilde v(y,0)-\tilde v(\tilde{x}_0,t_{0})+\tilde C_{0}h(y)+\tilde A|y-\tilde{x}_0|^{2})\\
&\geq 0.
\end{aligned}
\end{equation*}
By (\ref{e1.4}) and (\ref{eqq2.13aa}), it is not difficult to show that
$$\tilde L\tilde\Phi(y)\leq \left(C-\frac{\tilde C_0\theta}{2}+2\tilde A\right)\sum_{i=1}^n \tilde F^{ii}+2 \tilde A f_{p_i}(y_i-\tilde{x}_{0i})-\tilde C_0\left(\frac{\theta}{2}\sum_{i=1}^n \tilde F^{ii}-f_{p_i}\partial_i h\right). $$
In order to make
$$\tilde L\tilde\Phi(y)\leq 0,$$
we only need to choose $\tilde C_0\gg \tilde A$ and
$$|D f|\leq \frac{\theta\Lambda_1}{2}\cdot\frac{1}{\max_{\bar{\tilde \Omega}}|D h|}.$$
Consequently,
\begin{equation}\label{eqq2.13a}
\left\{ \begin{aligned}
   \tilde L\tilde \Phi&\leq 0,\ \  &&(y,t)\in\tilde\Omega_{r}\times[0,T],\\
   \tilde\Phi&\geq 0,\ \  &&(y,t)\in(\partial\tilde\Omega_{r}\times[0,T])\cup(\tilde\Omega_{r}\times\{t=0\}).
                          \end{aligned} \right.
\end{equation}
Therefore, we get (\ref{eeee3.13}) as same as the argument in (\ref{eeee3.8}). Thus we complete the proof of the lemma.
\end{proof}

Similar to Proposition 2.6 in \cite{SM}, by making use of (\ref{eeee3.15}) we can obtain
\begin{lemma}\label{llll3.0}
Fix a smooth function $H: \Omega\times\tilde{\Omega}\rightarrow R$ and define $\varphi(x,t)=H(x,Du(x,t))$. Then for any $(x,t)\in \Omega_{T}$,
\begin{equation*}
|L\varphi|\leq C\sum_{i=1}^n F^{ii}
\end{equation*}
holds for some positive constant $C$, which depends only on $H$, $\Omega$, $\tilde\Omega$, $\Lambda_1$, $\Lambda_2$ and $\delta$.
\end{lemma}

The following definition provides a basic connection between (\ref{eeee2.1.1001}) and (\ref{eeee2.1}) and will be used frequently in the sequel.
\begin{deff}\label{d1.10}
We say that $\tilde{u}$ in (\ref{eeee2.1.1001}) is a dual solution to (\ref{eeee2.1}).
\end{deff}
We now proceed to carry out the global $C^2$ estimate. The strategy is to  reduce the $C^2$ global estimate of $u$ and $\tilde{u}$ to the boundary.
\begin{lemma}\label{11113.5}
If $u$ is a strictly convex solution of (\ref{eeee2.1}) and there hold (\ref{e1.2}), (\ref{e1.2.1}) and (\ref{e1.3}), then there exists a positive constant $C$ depending only on $n$, $\Omega$, $\tilde{\Omega}$, $\Lambda_1$, $u_0$, $\delta$ and $\operatorname{diam}(\Omega)$, such that
\begin{equation}\label{ee3.170}
\sup_{\Omega_T}| D^{2}u|\leq \max_{\partial\Omega\times[0,T]}| D^{2}u|+\max_{\bar{\Omega}}| D^{2}u_{0}|+C\sup_{\Omega}|D^2f|.
\end{equation}
\end{lemma}

\begin{proof}
Without loss of generality, we may assume that $\Omega$ lies in cube $[0,d]^{n}$. Let
$$L:=F^{ij}\partial_{ij}-\partial_t.$$
For any unit vector $\xi$, differentiating the equation in (\ref{eeee2.1}) twice in direction $\xi$ gives
\begin{equation*}
Lu_{\xi\xi}+F^{ij,rs}u_{ij\xi}u_{rs\xi}=f_{\xi\xi}.
\end{equation*}
Then by the concavity of $F$ on $\Gamma^+_n$, we have
\begin{equation*}
Lu_{\xi\xi}=-F^{ij,rs}u_{ij\xi}u_{rs\xi}+f_{\xi\xi}\geq f_{\xi\xi}.
\end{equation*}

Let $$v=\sup_{\partial\Omega_T}u_{\xi\xi}+\frac{1}{\Lambda_1}\left(ne^{ d}-\sum^{n}_{i=1}e^{ x_{i}}\right)\sup_{\Omega}|f_{\xi\xi}|.$$
By direct calculation and (\ref{e1.3}), we obtain
\begin{equation*}
\begin{aligned}
F^{ij}\partial_{ij}v=&-\frac{1}{\Lambda_1}\sup_{\Omega}|f_{\xi\xi}|\left(\sum^{n}_{i=1}e^{ x_{i}}F^{ii}\right)\\
\leq &-\frac{1}{\Lambda_1}\sup_{\Omega}|f_{\xi\xi}|\left(\sum^{n}_{i=1}F^{ii}\right)\\
\leq &-\sup_{\Omega}|f_{\xi\xi}|.
\end{aligned}
\end{equation*}
Therefore,
$$Lv=F^{ij}\partial_{ij}v-\partial_t v\leq -\sup_{\Omega}|f_{\xi\xi}|,$$
and thus
$$L(v-u_{\xi\xi})\leq -\left(\sup_{\Omega}|f_{\xi\xi}|+f_{\xi\xi}\right)\leq 0.$$
It is obvious that $v-u_{\xi\xi}\geq 0$ on $\partial\Omega_T$. Then by the maximum principle we obtain
\begin{equation*}
\begin{aligned}
\sup_{\Omega_T}u_{\xi\xi}&\leq \sup_{\Omega_T}v\leq \sup_{\partial\Omega_T}u_{\xi\xi}+\frac{ne^{d}}{\Lambda_1}\sup_{\Omega}|f_{\xi\xi}|\\
&\leq \max_{\partial\Omega\times[0,T]}| D^{2}u|+\max_{\bar{\Omega}}| D^{2}u_{0}|+C\sup_{\Omega}|D^2f|.
\end{aligned}
\end{equation*}
This completes the proof of (\ref{ee3.170}).
\end{proof}

Next, we estimate the second order derivative on the boundary. By differentiating the boundary condition $h(Du)=0$ in any tangential direction $\varsigma$, we have
\begin{equation}\label{eq3.2}
   u_{\beta \varsigma}=h_{p_k}(Du)u_{k\varsigma}=0.
\end{equation}
The second order derivative of $u$ on the boundary is controlled by $u_{\beta \varsigma}$, $u_{\beta \beta}$ and $u_{\varsigma\varsigma}$. In the following we give the arguments as in \cite{JU}, one can see there for more details.

At $x\in \partial\Omega$, any unit vector $\xi$ can be written in terms of a tangential component $\varsigma(\xi)$ and a component in the direction $\beta$ by
$$\xi=\varsigma(\xi)+\frac{\langle \nu,\xi\rangle}{\langle\beta,\nu\rangle}\beta,$$
where
$$\varsigma(\xi):=\xi-\langle \nu,\xi\rangle \nu-\frac{\langle \nu,\xi\rangle}{\langle\beta,\nu\rangle}\beta^T,$$
and
$$\beta^T:=\beta-\langle \beta,\nu\rangle \nu.$$
By the strict obliqueness estimate (\ref{eeee3.4}), we have
\begin{equation}\label{eq3.3}
\begin{aligned}
|\varsigma(\xi)|^{2}&=1-\left(1-\frac{|\beta^{T}|^{2}}{\langle\beta,\nu\rangle^{2}}\right)\langle\nu,\xi\rangle^{2}
-2\langle\nu,\xi\rangle\frac{\langle\beta^{T},\xi\rangle}{\langle\beta,\nu\rangle}\\
&\leq 1+C\langle\nu,\xi\rangle^{2}-2\langle\nu,\xi\rangle\frac{\langle\beta^{T},\xi\rangle}{\langle\beta,\nu\rangle}\\
&\leq C.
\end{aligned}
\end{equation}
Denote $\varsigma:=\frac{\varsigma(\xi)}{|\varsigma(\xi)|}$, then by (\ref{eq3.3}) and (\ref{eeee3.4}) we obtain
\begin{equation}\label{eq3.4}
\begin{aligned}
u_{\xi\xi}&=|\varsigma(\xi)|^{2}u_{\varsigma\varsigma}+2|\varsigma(\xi)|\frac{\langle\nu,\xi\rangle}{\langle\beta,\nu\rangle}u_{\beta\varsigma}+
\frac{\langle\nu,\xi\rangle^{2}}{\langle\beta,\nu\rangle^{2}}
u_{\beta\beta}\\
&=|\varsigma(\xi)|^{2}u_{\varsigma\varsigma}+\frac{\langle\nu,\xi\rangle^{2}}{\langle\beta,\nu\rangle^{2}}
u_{\beta\beta}\\
&\leq C(u_{\varsigma\varsigma}+u_{\beta\beta}),
\end{aligned}
\end{equation}
where $C$ depends only on $\Omega$, $\tilde \Omega$, $\Lambda_1$, $\Lambda_2$, $\delta$ and the constant $C_1$ in (\ref{eeee3.4}). Therefore, we only need to estimate $u_{\beta\beta}$ and $u_{\varsigma\varsigma}$ respectively.

Further we have
\begin{lemma}\label{lem3.2}
Let $F$ satisfy the structure conditions (\ref{e1.2.0})-(\ref{e1.4}) and $f\in {\mathscr{A}}_\delta$. If $u$ is a strictly convex solution of (\ref{eeee2.1}), then there exists a positive constant $C$ depending only on $u_0$, $\Omega$, $\tilde{\Omega}$, $\Lambda_1$, $\Lambda_2$ and $\delta$, such that
\begin{equation}\label{eq3.5}
\max_{\partial\Omega_T}u_{\beta\beta} \leq C.
\end{equation}
\end{lemma}

\begin{proof}
Let $x_0\in\partial\Omega$, $t_0\in [0,T]$ satisfy $u_{\beta\beta}(x_0,t_0)=\max_{\partial\Omega_T}u_{\beta\beta}$. Consider the barrier function
$$\Psi:=-h(Du)+C_0\tilde{h}+A|x-x_0|^2.$$
For any $x\in \partial\Omega$, $Du(x)\in \partial\tilde{\Omega}$, then $h(Du)=0$. It is clear that $\tilde{h}=0$ on $\partial\Omega$. As same as  the proof of (\ref{eeee3.16}), we can find the constants $C_0$ and $A$ such that
\begin{equation}\label{eq3.6}
\left\{ \begin{aligned}L\Psi&\leq 0,\qquad&&(x,t)\in\Omega_{\rho}\times[0,T] , \\
\Psi&\geq 0,\qquad &&(x,t)\in(\partial\Omega_{\rho}\times[0,T])\cup(\Omega_{\rho}\times\{t=0\}).
\end{aligned} \right.
\end{equation}
By the maximum principle, we get
$$\Psi\geq 0,\quad (x,t)\in\Omega_{\rho}\times[0,T].$$
Combining it with $\Psi(x_{0},t_{0})=0$ we obtain $\Psi_{\beta}(x_{0},t_{0})\geq 0$, which implies
$$\frac{\partial h}{\partial \beta}(Du(x_0,t_0))\leq C_0.$$

On the other hand, we see that at $(x_0,t_0)$,
$$\frac{\partial h}{\partial \beta}=\langle Dh(Du),\beta\rangle=\frac{\partial h}{\partial p_k}u_{kl}\beta^l=\beta^ku_{kl}\beta^l=u_{\beta\beta}.$$
Therefore,
$$u_{\beta\beta}=\frac{\partial h}{\partial \beta}\leq C,$$whence the result follows.
\end{proof}

Next, we estimate the double tangential derivative.
\begin{lemma}\label{lem3.3}
Let $F$ satisfy the structure conditions (\ref{e1.2.0})-(\ref{e1.4}) and $f\in {\mathscr{A}}_\delta$. If $u$ is a strictly convex solution of (\ref{eeee2.1}), then there exists a positive constant $C$ depending only on $u_0$, $\Omega$, $\tilde{\Omega}$, $\Lambda_1$, $\Lambda_2$ and $\delta$, such that
\begin{equation}\label{eq3.7}
\max_{\partial\Omega\times[0,T]}\max_{|\varsigma|=1, \langle\varsigma,\nu\rangle=0} u_{\varsigma\varsigma}\leq C.
\end{equation}
\end{lemma}

\begin{proof}
Assume that $x_{0}\in\partial\Omega$, $t_{0}\in[0,T]$ and $e_{n}$ is the unit inward normal vector of $\partial\Omega$ at $x_0$. Let
$$\max_{\partial\Omega\times[0,T]}\max_{|\varsigma|=1, \langle\varsigma,\nu\rangle=0} u_{\varsigma\varsigma}=u_{11}(x_{0},t_{0})=:M.$$

For any $x\in \partial\Omega$,  we have by (\ref{eq3.3}),
\begin{equation}\label{eq3.8}
\begin{aligned}
u_{\xi\xi}&=|\varsigma(\xi)|^2u_{\varsigma\varsigma}+ \frac{\langle \nu,\xi\rangle^2}{\langle \beta,\nu\rangle^2}u_{\beta\beta}\\
          &\leq \left(1+C\langle \nu,\xi\rangle^2-2\langle \nu,\xi\rangle \frac{\langle \beta^T,\xi\rangle}{\langle \beta,\nu\rangle}\right)M
                 + \frac{\langle \nu,\xi\rangle^2}{\langle \beta,\nu\rangle^2}u_{\beta\beta}.
\end{aligned}
\end{equation}
Without loss of generality, we assume that $M\geq 1$. Then by (\ref{eeee3.4}) and (\ref{eq3.5}) we have
\begin{equation}\label{eq3.9}
  \frac{u_{\xi\xi}}{M}+2\langle \nu,\xi\rangle \frac{\langle \beta^T,\xi\rangle}{\langle \beta,\nu\rangle}
             \leq 1+C\langle \nu,\xi\rangle^2.
\end{equation}
Let $\xi=e_1$, then
\begin{equation}\label{eq3.10}
  \frac{u_{11}}{M}+2\langle \nu,e_1\rangle \frac{\langle \beta^T,e_1\rangle}{\langle \beta,\nu\rangle}
             \leq 1+C\langle \nu,e_1\rangle^2.
\end{equation}
We see that the function
\begin{equation}\label{eq3.11}
w:=A|x-x_0|^2-\frac{u_{11}}{M}-2\langle \nu,e_1\rangle \frac{\langle \beta^T,e_1\rangle}{\langle \beta,\nu\rangle}+C\langle \nu,e_1\rangle^2+1
\end{equation}
satisfies
$$w|_{\partial\Omega\times[0,T]}\geq 0, \quad w(x_{0},t_{0})=0.$$
Then, it follows by (\ref{ee3.170}) that we can choose the constant $A$ large enough such that
$$w|_{(\partial B_{\rho}(x_{0})\cap\Omega)\times[0,T]}\geq 0.$$

Consider
$$-2\langle\nu,e_{1}\rangle\frac{\langle\beta^{T},e_{1}\rangle}{\langle\beta,\nu\rangle}+C\langle\nu,e_{1}\rangle^{2}+1$$
as a known function depending on $x$ and $Du$. Then by Lemma \ref{llll3.0}, we obtain
$$\left|L\left(-2\langle\nu,e_{1}\rangle\frac{\langle\beta^{T},e_{1}\rangle}{\langle\beta,\nu\rangle}
+C\langle\nu,e_{1}\rangle^{2}+1\right)\right|\leq C\sum_{i=1}^n F^{ii}.$$
Combining the above inequality with the proof of Lemma \ref{11113.5},  we have
$$Lw\leq C\sum_{i=1}^n F^{ii}.$$

As in the proof of Lemma \ref{lem3.2}, we consider the function
$$\Upsilon:=w+C_0\tilde{h}.$$
A standard barrier argument shows that
$$\Upsilon_{\beta}(x_0,t_0)\geq0.$$
Therefore,
\begin{equation}\label{eq3.12}
 u_{11\beta}(x_0,t_0)\leq CM.
\end{equation}
On the other hand, differentiating $h(Du)$ twice in the direction $e_1$ at $(x_0,t_0)$, we have
$$h_{p_k}u_{k11}+h_{p_kp_l}u_{k1}u_{l1}=0.$$
The concavity of $h$ yields that
$$h_{p_k}u_{k11}=-h_{p_kp_l}u_{k1}u_{l1}\geq \theta M^2.$$
Combining it with $h_{p_k}u_{k11}=u_{11\beta}$, and using (\ref{eq3.12}) we obtain
$$\theta M^2\leq CM.$$
Then we get the upper bound of $M=u_{11}(x_0,t_0)$ and thus the desired result follows.
\end{proof}

By Lemma \ref{lem3.2}, Lemma \ref{lem3.3} and (\ref{eq3.4}), we obtain the $C^2$ a-priori estimate on the boundary.
\begin{lemma}\label{lem3.4}
Let $F$ satisfy the structure conditions (\ref{e1.2.0})-(\ref{e1.4}) and $f\in {\mathscr{A}}_\delta$. If $u$ is a strictly convex solution of (\ref{eeee2.1}), then there exists a positive constant $C$ depending only on $u_0$, $\Omega$, $\tilde{\Omega}$, $\Lambda_1$, $\Lambda_2$ and $\delta$, such that
\begin{equation}\label{eq3.13}
\max_{\partial\Omega_T}|D^2u| \leq C.
\end{equation}
\end{lemma}

In terms of Lemma \ref{11113.5} and Lemma \ref{lem3.4}, we  readily conclude:
\begin{lemma}\label{lem3.5}
Let $F$ satisfy the structure conditions (\ref{e1.2.0})-(\ref{e1.4}) and $f\in {\mathscr{A}}_\delta$. If $u$ is a strictly convex solution of (\ref{eeee2.1}), then there exists a positive constant $C$ depending only on $u_0$, $\Omega$, $\tilde{\Omega}$, $\Lambda_1$, $\Lambda_2$ and $\delta$, such that
\begin{equation}\label{eq3.14}
\max_{\bar{\Omega}_T}|D^2u| \leq C.
\end{equation}
\end{lemma}

In the following,  we describe the positive lower bound of $D^{2}u$. For (\ref{eeee2.1.1001}), in consider of the Legendre transformation of $u$, define
$$\tilde L:= \tilde F^{ij}\partial_{ij}+f_{p_i}\partial_i-\partial_t.$$
Then our goal is to show the upper bound of $D^{2}\tilde{u}$ and the argument is very similar to the one used in the proof of Lemma \ref{lem3.5} by the concavity of $f$ and the condition that $|D f|$ being sufficiently small. For the convenience of readers, we give the details.

At the beginning of  the repeating procedure, we have
\begin{lemma}\label{lem3.1-1}
Suppose that $f$ is concave on $\Omega$. If $\tilde{u}$ is a strictly convex solution of (\ref{eeee2.1.1001}), then there holds
\begin{equation}\label{eq3.1a}
   \sup_{\tilde{\Omega}_T}|D^2\tilde{u}|\leq \max_{\partial\tilde{\Omega}_T} |D^2\tilde{u}|.
\end{equation}
\end{lemma}
\begin{proof}
For any unit vector $\tilde\xi$, differentiating the equation in (\ref{eeee2.1.1001}) twice in direction $\tilde\xi$ gives
\begin{equation*}
\tilde{L}\tilde{u}_{\tilde\xi\tilde\xi}+\tilde{F}^{ij,rs}\tilde{u}_{ij\tilde\xi}\tilde{u}_{rs\tilde\xi}+\frac{\partial^{2}f}{\partial p_{i}\partial p_{j}}\tilde{u}_{i\tilde\xi}\tilde{u}_{j\tilde\xi}=0.
\end{equation*}
Then by the concavity of $\tilde{F}$ on $\Gamma^+_n$ and $f$ on $\Omega$, we have
\begin{equation*}
\tilde{L}\tilde{u}_{\tilde\xi\tilde\xi}=-\tilde{F}^{ij,rs}\tilde{u}_{ij\tilde\xi}\tilde{u}_{rs\tilde\xi}-\frac{\partial^{2}f}{\partial p_{i}\partial p_{j}}\tilde{u}_{i\tilde\xi}\tilde{u}_{j\tilde\xi}\geq 0.
\end{equation*}
Then by the maximum principle we obtain
$$\sup_{\tilde{\Omega}_T}\tilde{u}_{\tilde\xi\tilde\xi}\leq \sup_{\partial\tilde{\Omega}_T}\tilde{u}_{\tilde\xi\tilde\xi}.$$
This completes the proof of (\ref{eq3.1a}).
\end{proof}

Recall that $\tilde{\beta}=(\tilde{\beta}^{1}, \cdots, \tilde{\beta}^{n})$ with $\tilde{\beta}^{k}:=\tilde{h}_{p_{k}}(D\tilde{u})$ and $\tilde{\nu}=(\tilde{\nu}_{1}, \tilde{\nu}_{2},\cdots,\tilde{\nu}_{n})$ is the unit inward normal vector of $\partial\tilde{\Omega}$. Similar to the discussion of (\ref{eq3.2}), (\ref{eq3.3}) and (\ref{eq3.4}), for any tangential direction $\tilde \varsigma$, we have
\begin{equation}\label{eq3.2a}
   u_{\tilde\beta \tilde\varsigma}=\tilde h_{p_k}(D\tilde u)\tilde u_{k\tilde\varsigma}=0.
\end{equation}
Then the second order derivative of $\tilde u$ on the boundary is also controlled by $u_{\tilde\beta \tilde\varsigma}$, $u_{\tilde\beta \tilde\beta}$ and $u_{\tilde\varsigma\tilde\varsigma}$.

At $\tilde x\in \partial\tilde\Omega$, any unit vector $\tilde\xi$ can be written in terms of a tangential component $\tilde\varsigma(\tilde\xi)$ and a component in the direction $\tilde\beta$ by
$$\tilde\xi=\tilde\varsigma(\tilde\xi)+\frac{\langle \tilde\nu,\tilde\xi\rangle}{\langle\tilde\beta,\tilde\nu\rangle}\tilde\beta,$$
where
$$\tilde\varsigma(\tilde\xi):=\tilde\xi-\langle \tilde\nu,\tilde\xi\rangle \tilde\nu-\frac{\langle \tilde\nu,\tilde\xi\rangle}{\langle\tilde\beta,\tilde\nu\rangle}\tilde\beta^T,$$
and
$$\tilde\beta^T:=\tilde\beta-\langle \tilde\beta,\tilde\nu\rangle \tilde\nu.$$
We observe that $\langle\tilde\beta,\tilde\nu\rangle=\langle\beta,\nu\rangle$.  Therefore,
\begin{equation}\label{eq3.3a}
|\tilde\varsigma(\tilde\xi)|\leq C,
\end{equation}
and similar to the calculation in (\ref{eq3.8}), one should deduce that
\begin{equation}\label{eq3.4a}
\tilde{u}_{\tilde\xi\tilde\xi}\leq C(\tilde{u}_{\tilde\varsigma\tilde\varsigma}+\tilde{u}_{\tilde\beta\tilde\beta}),
\end{equation}
where $\tilde\varsigma:=\frac{\tilde\varsigma(\tilde\xi)}{|\tilde\varsigma(\tilde\xi)|}$ and $C$ depends only on $\Omega$, $\tilde \Omega$, $\Lambda_1$, $\Lambda_2$, $\delta$ and the constant $C_1$ in (\ref{eeee3.4}). Then we also only need to estimate $\tilde{u}_{\tilde\beta\tilde\beta}$ and $\tilde{u}_{\tilde\varsigma\tilde\varsigma}$ respectively.

Indeed, as shown by Lemma \ref{lem3.2}, we state
\begin{lemma}\label{lem3.2a}
Let $F$ satisfy the structure conditions (\ref{e1.2.0})-(\ref{e1.4}) and $f\in {\mathscr{A}}_\delta$. If $\tilde{u}$ is a strictly convex solution of (\ref{eeee2.1.1001}) and $|D f|$ is sufficiently small, then there exists a positive constant $C$ depending only on $u_0$, $\Omega$, $\tilde{\Omega}$, $\Lambda_1$, $\Lambda_2$ and $\delta$, such that
\begin{equation}\label{eq3.5a}
   \max_{\partial\Omega_T}\tilde{u}_{\tilde{\beta}\tilde{\beta}} \leq C.
\end{equation}
\end{lemma}

\begin{proof}
Let $\tilde{x}_0\in\partial\tilde{\Omega}$, $t_0\in[0,T]$ satisfy $\tilde{u}_{\tilde{\beta}\tilde{\beta}}(\tilde{x}_0,t_0)=\max_{\partial\Omega_T}\tilde{u}_{\tilde{\beta}\tilde{\beta}}$.
To estimate the upper bound of $\tilde{u}_{\tilde{\beta}\tilde{\beta}}$,
we consider the barrier function
$$\tilde\Psi:=-\tilde{h}(D\tilde{u})+C_0 h+A|y-\tilde{x}_0|^2.$$
For any $y\in \partial\tilde{\Omega}$, $D\tilde{u}(y)\in \partial\Omega$, then $\tilde{h}(D\tilde{u})=0$. It is clear that $h=0$ on $\partial\tilde{\Omega}$. As the proof of (\ref{eqq2.13a})  in terms of $|D f|$ being sufficiently small, we can find the constants $C_0$ and $A$ such that
\begin{equation}\label{eq3.6a}
\left\{ \begin{aligned}
   \tilde{L}\tilde\Psi&\leq 0,\quad &&(y,t)\in\tilde\Omega_{r}\times[0,T],\\
   \tilde\Psi&\geq 0 ,\quad\  &&(y,t)\in(\partial\tilde\Omega_{r}\times[0,T])\cup(\tilde\Omega_{r}\times\{t=0\}).
\end{aligned} \right.
\end{equation}
By the maximum principle, we get
$$\tilde\Psi(y,t)\geq 0,\quad\quad (y,t)\in \tilde{\Omega}_r\times[0,T].$$
Combining it with $\tilde\Psi(\tilde{x}_0,t_0)=0$ we obtain $\tilde\Psi_{\tilde{\beta}}(\tilde{x}_0,t_0)\geq 0$, which implies
$$\frac{\partial \tilde{h}}{\partial \tilde{\beta}}(D\tilde{u}(\tilde{x}_0,t_0))\leq C_0.$$

On the other hand, we see that at $(\tilde{x}_0,t_0)$,
$$\frac{\partial \tilde{h}}{\partial \tilde{\beta}}=\langle D\tilde{h}(D\tilde{u}),\tilde{\beta}\rangle=\frac{\partial \tilde{h}}{\partial p_k}\tilde{u}_{kl}\tilde{\beta}^l=\tilde{\beta}^k\tilde{u}_{kl}\tilde{\beta}^l=\tilde{u}_{\tilde{\beta}\tilde{\beta}}.$$
Therefore,
$$\tilde{u}_{\tilde{\beta}\tilde{\beta}}=\frac{\partial \tilde{h}}{\partial \tilde{\beta}}\leq C.$$
\end{proof}

Next, we estimate the double tangential derivative of $\tilde{u}$.
\begin{lemma}\label{lem3.3a}
Let $F$ satisfy the structure conditions (\ref{e1.2.0})-(\ref{e1.4}) and $f\in {\mathscr{A}}_\delta$. If $\tilde{u}$ is a strictly convex solution of (\ref{eeee2.1.1001}) and $|D f|$ is sufficiently small, then there exists a positive constant $C$ depending only on $u_0$, $\Omega$, $\tilde{\Omega}$, $\Lambda_1$, $\Lambda_2$ and $\delta$, such that
\begin{equation}\label{eq3.7a}
   \max_{\partial\tilde{\Omega}\times[0,T]}\max_{|\tilde\varsigma|=1, \langle\tilde\varsigma,\tilde\nu\rangle=0} \tilde{u}_{\tilde\varsigma\tilde\varsigma} \leq C.
\end{equation}
\end{lemma}

\begin{proof}
Assume that $\tilde{x}_0\in\partial \tilde{\Omega}$, $t_{0}\in[0,T]$ and $e_{n}$ is the unit inward normal vector of $\partial\tilde{\Omega}$ at $\tilde{x}_0$. Let
$$ \max_{\partial\tilde{\Omega}\times[0,T]}\max_{|\tilde\varsigma|=1, \langle\tilde\varsigma,\tilde\nu\rangle=0} \tilde{u}_{\tilde\varsigma\tilde\varsigma}=\tilde{u}_{11}(\tilde{x}_0,t_0)=:\tilde M.$$

For any $y\in \partial\tilde{\Omega}$,  we have by (\ref{eq3.3a}),
\begin{equation}\label{eq3.8a}
\begin{aligned}
\tilde u_{\tilde\xi\tilde\xi}&=|\tilde\varsigma(\tilde\xi)|^2\tilde u_{\tilde\varsigma\tilde\varsigma}+ \frac{\langle \tilde\nu,\tilde\xi\rangle^2}{\langle \tilde\beta,\tilde\nu\rangle^2}\tilde u_{\tilde\beta\tilde\beta}\\
          &\leq \left(1+C\langle \tilde\nu,\tilde\xi\rangle^2-2\langle \tilde\nu,\tilde\xi\rangle \frac{\langle \tilde\beta^T,\tilde\xi\rangle}{\langle \tilde\beta,\tilde\nu\rangle}\right)\tilde M
                 + \frac{\langle \tilde\nu,\tilde\xi\rangle^2}{\langle \tilde\beta,\tilde\nu\rangle^2}\tilde u_{\tilde\beta\tilde\beta}.
\end{aligned}
\end{equation}
Without loss of generality, we assume that $\tilde M\geq 1$. Then by (\ref{eeee3.4}) and (\ref{eq3.5a}) we have
\begin{equation}\label{eq3.9a}
  \frac{\tilde u_{\tilde\xi\tilde\xi}}{\tilde M}+2\langle \tilde\nu,\tilde\xi\rangle \frac{\langle \tilde\beta^T,\tilde\xi\rangle}{\langle \tilde\beta,\tilde\nu\rangle}
             \leq 1+C\langle \tilde\nu,\tilde\xi\rangle^2.
\end{equation}
Let $\tilde\xi=e_1$, then
\begin{equation}\label{eq3.10a}
  \frac{\tilde u_{11}}{\tilde M}+2\langle \tilde\nu,e_1\rangle \frac{\langle \tilde\beta^T,e_1\rangle}{\langle \tilde\beta,\tilde\nu\rangle}
             \leq 1+C\langle \tilde\nu,e_1\rangle^2.
\end{equation}
We see that the function
\begin{equation}\label{eq3.11a}
\tilde w:=A|y-\tilde x_0|^2-\frac{\tilde u_{11}}{\tilde M}-2\langle \tilde\nu,e_1\rangle \frac{\langle \tilde\beta^T,e_1\rangle}{\langle \tilde\beta,\tilde\nu\rangle}+C\langle \tilde\nu,e_1\rangle^2+1
\end{equation}
satisfies
$$\tilde w|_{\partial\tilde\Omega\times[0,T]}\geq 0,\quad  \tilde w(\tilde x_0,t_0)=0.$$
Then, by (\ref{eq3.1a}) we can choose the constant $A$ large enough such that
$$\tilde w|_{(\tilde\Omega \cap \partial B_{r}(\tilde x_0))\times[0,T]} \geq 0.$$

Consider
$$-2\langle \tilde\nu,e_1\rangle \frac{\langle \tilde\beta^T,e_1\rangle}{\langle \tilde\beta,\tilde\nu\rangle}+C\langle \tilde\nu,e_1\rangle^2+1$$
as a known function depending on $\tilde x$ and $D\tilde u$. Then by Lemma \ref{llll3.0}, we obtain
$$\left|\tilde L\left(-2\langle \tilde\nu,e_1\rangle \frac{\langle \tilde\beta^T,e_1\rangle}{\langle \tilde\beta,\tilde\nu\rangle}+C\langle \tilde\nu,e_1\rangle^2+1\right)\right|\leq C\sum_{i=1}^n \tilde F^{ii}.$$
Combining the above inequality with the proof of Lemma \ref{lem3.1-1}, by $f\in {\mathscr{A}}_\delta$ and $|D f|$ is sufficiently small, we have
$$\tilde L\tilde w\leq  C\sum^{n}_{i=1} \tilde F^{ii}.$$

As in the proof of Lemma \ref{lem3.2a}, consider the function
$$\tilde\Upsilon:=\tilde w+C_0 {h}.$$
A standard barrier argument makes conclusion of
$$\tilde\Upsilon_{\tilde\beta}(\tilde x_0,t_0)\geq0.$$
Therefore,
\begin{equation}\label{eq3.12a}
 \tilde u_{11\tilde\beta}(\tilde x_0)\leq C \tilde M.
\end{equation}
On the other hand, differentiating $\tilde h(D\tilde u)$ twice in the direction $e_1$ at $(\tilde x_0,t_0)$, we have
$$\tilde h_{p_k}\tilde u_{k11}+\tilde h_{p_kp_l}\tilde u_{k1}\tilde u_{l1}=0.$$
The concavity of $\tilde h$ yields that
$$\tilde h_{p_k}\tilde u_{k11}=-\tilde h_{p_kp_l}\tilde u_{k1}\tilde u_{l1}\geq \tilde\theta \tilde M^2.$$
Combining it with $\tilde h_{p_k}\tilde u_{k11}=\tilde u_{11\tilde\beta}$, and using (\ref{eq3.12a}) we obtain
$$\tilde\theta \tilde M^2\leq C\tilde M.$$
Then we get the upper bound of $\tilde M=\tilde u_{11}(\tilde x_0,t_0)$ and thus the desired result follows.
\end{proof}

By Lemma \ref{lem3.2a}, Lemma \ref{lem3.3a} and (\ref{eq3.4a}), we obtain the $C^2$ a-priori estimate of $\tilde{u}$ on the boundary.
\begin{lemma}\label{lem3.4a}
Let $F$ satisfy the structure conditions (\ref{e1.2.0})-(\ref{e1.4}) and $f\in {\mathscr{A}}_\delta$. If $\tilde{u}$ is a strictly convex solution of (\ref{eeee2.1.1001}) and $|D f|$ is sufficiently small, then there exists a positive constant $C$ depending only on $u_0$, $\Omega$, $\tilde{\Omega}$, $\Lambda_1$, $\Lambda_2$ and $\delta$, such that
\begin{equation}\label{eq3.13a}
\max_{\partial\tilde\Omega_T}|D^2\tilde u| \leq C.
\end{equation}
\end{lemma}

By Lemma \ref{lem3.1-1} and Lemma \ref{lem3.4a}, we can see that
\begin{lemma}\label{lem3.5a}
Let $F$ satisfy the structure conditions (\ref{e1.2.0})-(\ref{e1.4}) and $f\in {\mathscr{A}}_\delta$. If $\tilde{u}$ is a strictly convex solution of (\ref{eeee2.1.1001}) and $|D f|$ is sufficiently small, then there exists a positive constant $C$ depending only on $u_0$, $\Omega$, $\tilde{\Omega}$, $\Lambda_1$, $\Lambda_2$ and $\delta$, such that
\begin{equation}\label{eq3.14a}
\max_{\bar{\tilde\Omega}_T}|D^2\tilde u| \leq C.
\end{equation}
\end{lemma}

By Lemma \ref{lem3.5} and Lemma \ref{lem3.5a}, we conclude that
\begin{lemma}\label{lem3.6}
Let $F$ satisfy the structure conditions (\ref{e1.2.0})-(\ref{e1.4}) and $f\in {\mathscr{A}}_\delta$. If $u$ is a strictly convex solution of (\ref{eeee2.1}) and $|D f|$ is sufficiently small, then there exists a positive constant $C$ depending only on $u_0$, $\Omega$, $\tilde{\Omega}$, $\Lambda_1$, $\Lambda_2$ and $\delta$, such that
\begin{equation}\label{eq3.15}
\frac{1}{C}I_n\leq D^2 u(x,t) \leq C I_n,\ \ (x,t)\in\bar\Omega_T,
\end{equation}
where $I_n$ is the $n\times n$ identity matrix.
\end{lemma}

\section{Longtime existence and convergence}
We will need the following proposition, which essentially asserts the convergence of the flow.
\begin{Proposition}\label{pppp5.1}
(Huang and Ye, see Theorem 1.1 in  \cite{HRY1}.) For any $T>0$, we assume that $u\in C^{4+\alpha,\frac{4+\alpha}{2}}(\bar{\Omega}_{T})$ is a unique solution of the nonlinear parabolic equation (\ref{eeee2.1}), which satisfies
\begin{equation}\label{e51.4}
\|u_{t}(\cdot,t)\|_{C(\bar{\Omega})}+\|Du(\cdot,t)\|_{C(\bar{\Omega})}+\|D^{2}u(\cdot,t)\|_{C(\bar{\Omega})}\leq C_{2},
\end{equation}
\begin{equation}\label{e51.40}
\|D^{2}u(\cdot,t)\|_{C^{\alpha}(\bar{D})}\leq C_{3},\quad \forall D\subset\subset\Omega,
\end{equation}
and
\begin{equation}\label{e51.5}
\inf_{x \in \partial \Omega}\big (\sum_{k=1}^{n}h_{p_{k}}(Du(x,t))\nu_{k} \big ) \geq \frac{1}{C_{4}},
\end{equation}
where the positive constants $C_{1}$, $C_{2}$ and $C_{3}$  are independent of ~ $t\geq 1$. Then the solution $u(\cdot,t)$ converges to a function $u^{\infty}(x,t)=\tilde{u}^\infty(x)+C_{\infty}\cdot t$
in $C^{1+\zeta}(\bar{\Omega})\cap C^{4}(\bar{D})$ as $t\rightarrow\infty$
  for any $D\subset\subset\Omega$, $\zeta<1$, that is
 $$\lim_{t\rightarrow+\infty}\|u(\cdot,t)-u^{\infty}(\cdot,t)\|_{C^{1+\zeta}(\bar{\Omega})}=0,\qquad
  \lim_{t\rightarrow+\infty}\|u(\cdot,t)-u^{\infty}(\cdot,t)\|_{C^{4}(\bar{D})}=0.$$
And $\tilde{u}^{\infty}(x)\in C^{2}(\bar{\Omega})$ is a solution of
\begin{equation}\label{e51.6}
\left\{ \begin{aligned}F(D^{2}u)-f(x)&=C_{\infty},
&  x\in \Omega, \\
h(Du)&=0, &x\in\partial\Omega.
\end{aligned} \right.
\end{equation}
The constant $C_{\infty}$ depends only on $\Omega$ $f$, and $F$. The solution to (\ref{e51.6}) is unique up to additions of constants.
\end{Proposition}

Now, we can give\\
\noindent{\bf Proof of Theorem \ref{tttt1.1}.}\\
This a standard result by our $C^{2}$ estimates and uniformly oblique estimates, but for convenience we include here a proof.\\

 Part 1: The long time existence.

By Lemma \ref{lem3.6}, we know the global $C^{2,1}$ estimates for the solutions of the flow (\ref{eeee1.1})-(\ref{eeee1.3}). Using Theorem 14.22 in Lieberman \cite{GM} and Lemma \ref{llll3.4}, we can show that the solutions of the oblique derivative problem \eqref{eeee2.1} have global $C^{2+\alpha,1+\frac{\alpha}{2}}$ estimates.

Now let $u_{0}$  be a $C^{2+\alpha_{0}}$ strictly convex function as in the conditions of Theorem \ref{tttt1.1}. We assume that $T$ is the maximal time such that the solution to the flow (\ref{eeee2.1}) exists. Suppose that $T<+\infty$. Combining Proposition \ref{pppp1.1} with Lemma \ref{lem3.6} and using Theorem 14.23 in \cite{GM}, there exists $ u\in C^{2+\alpha,1+\frac{\alpha}{2}}(\bar{\Omega}_{T})$ which satisfies (\ref{eeee2.1}) and
$$\|u\|_{C^{2+\alpha,1+\frac{\alpha}{2}}(\bar{\Omega}_{T})}<+\infty.$$
Then we can extend the flow (\ref{eeee2.1}) beyond the maximal time $T$. So that we deduce that $T=+\infty$. Then there exists the solution $u(x,t)$ for all times $t>0$ to (\ref{eeee1.1})-(\ref{eeee1.3}).\\

 Part 2: The convergence.

By the boundary condition, we have
\begin{equation*}
|Du|\leq C_{5},
\end{equation*}
where $C_{5}$ is a constant depending on $\Omega$ and $\tilde{\Omega}$. By intermediate Schauder estimates for parabolic
equations (cf. Lemma 14.6 and Proposition 4.25 in \cite{GM}), for any $D\subset\subset \Omega$, we have
\begin{equation*}
[D^{2}u]_{\alpha,\frac{\alpha}{2}, D_{T}}\leq C\sup|D^{2}u|\leq C_{6},
\end{equation*}
and
\begin{equation*}
\sup_{t\geq 1}\|D^{3}u(\cdot,t)\|_{C(\bar{D})}+\sup_{t\geq 1}\|D^{4}u(\cdot,t)\|_{C(\bar{D})}+\sup_{x_{i}\in D, t_{i}\geq 1}\frac{|D^{4}u(x_{1}, t_{1})-D^{4}u(x_{2}, t_{2})|}{\max\{|x_{1}-x_{2}|^{\alpha},|t_{1}-t_{2}|^{\frac{\alpha}{2}}\}}\leq C_{7},
\end{equation*}
where $ C_{6}$, $ C_{7}$ are  constants depending on the known data and dist$(\partial \Omega, \partial D)$.

Using Proposition \ref{pppp5.1} and combining the bootstrap arguments as in \cite{WHB}, we finish the proof of Theorem \ref{tttt1.1}.\qed\\

Finally, we can present\\
\noindent{\bf Proof of Theorem \ref{t1.1}.}\\
By Proposition \ref{prop2.10.2} and Remark \ref{rem2020}, we see that Theorem \ref{t1.1} is a direct consequence of Theorem \ref{tttt1.1}.\qed

\section*{Acknowledgements}
The authors would like to express deep gratitude to Professor Yuanlong Xin for his suggestions and constant encouragement.


\begin{thebibliography}{DU}
\bibitem{WHB}C. Wang, R.L. Huang, J.G. Bao, {\it On the second boundary value problem for Lagrangian mean curvature equation}, arXiv:1808.01139, 2018.


\bibitem{MW} M. Warren, {\it Calibrations associated to Monge-Amp\`{e}re equations}, Trans. Amer. Math. Soc. {\bf 362} (2010), 3947--3962.

\bibitem{HL}  R. Harvey, H.B. Lawson, {\it Calibrated geometry}. Acta Math. {\bf148} (1982), 47--157.

\bibitem{JX}J. Jost, Y.L. Xin, {\it A Bernstein theorem for special Lagrangian graphs}, Calc.Var.Partial Differential Equations. {\bf15} (2002), 299--312.
\bibitem{Y}Y. Yuan, {\it A Bernstein problem for special Lagrangian equations},   Invent. Math. {\bf150} (2002), 117--125.

\bibitem{P}P. Delano\"{e}, {\it Classical solvability in dimension two of the second boundary-value problem associated with the Monge-Amp\`{e}re operator,} Ann. Inst. H. Poincar$\acute{e}$ Anal. Non Lin$\acute{e}$aire.  {\bf 8} (1991), 443--457.

\bibitem{L} L. Caffarelli, {\it Boundary regularity of maps with convex potentials,} II, Ann. of Math. Stud.
{\bf 144} (1996), 453--496.

\bibitem{JU} J. Urbas, {\it On the second boundary value problems for equations of Monge-Amp\`{e}re type,} J. Reine Angew. Math. {\bf 487} (1997), 115--124.

\bibitem{OK} O.C. Schn\"{u}rer, K. Smoczyk, {\it Neumann and second boundary value problems for Hessian and  Gauss curvature flows,} Ann. Inst. H. Poincar$\acute{e}$ Anal. Non Lin$\acute{e}$aire. {\bf 20} (2003), 1043--1073.

\bibitem{SM}S. Brendle, M. Warren,
{\it A boundary value problem for minimal Lagrangian graphs}, J. Differential Geom. {\bf 84} (2010), 267--287.

\bibitem{HR} R.L. Huang, {\it On the second boundary value problem for Lagrangian mean curvature flow}, J. Funct. Anal. {\bf 269} (2015), 1095--1114.
\bibitem{HO}R.L. Huang, Q.Z. Ou, {\it On the second boundary value problem for a class of fully nonlinear equations}, J. Geom. Anal. {\bf 27} (2017), 2601--2617.

\bibitem{HRY}R.L. Huang, Y.H. Ye, {\it On the second boundary value problem for a class of fully nonlinear flows} I, Int. Math. Res. Not£¬{\bf 18} (2019), 5539-5576.

\bibitem{CHY} J.J. Chen, R.L. Huang, Y.H. Ye, {\it On the second boundary value problem for a class of fully nonlinear flows} II,
 Archiv Der Mathematik. {\bf 111} (2018), 407--419.



\bibitem{SL}
S.J. Altschuler, L.F. Wu,{\it
Translating surfaces of the non-parametric mean curvature flow with prescribed contact angle,}
Calc.Var.Partial Differential Equations. {\bf 2}(1994), 101--111.

\bibitem{OC} O.C. Schn$\ddot{\text{u}}$rer, {\it Translating solutions to the second boundary value problem
for curvature flows,} Manuscripta Math. {\bf108} (2002), 319--347.

\bibitem{JK} J. Kitagawa, {\it A parabolic flow toward solutions of the optimal transportation problem on domains with boundary,}
J. Rein. Angew. Math. {\bf672}(2012), 127--160.

\bibitem{W}Z.Q. Wu, J.X. Yin, and C.P. Wang, {\it Elliptic and parabolic equations,} World Scientific, 2006.


\bibitem{IE} I. Ekeland, {\it An inverse function theorem in Fr$\acute{e}$chet spaces}, Ann. Inst. H. Poincar$\acute{e}$ Anal. Non Lin$\acute{e}$aire. {\bf 28} (2011), 91-105.

\bibitem{GM} G.M. Lieberman, {\it Second order parabolic differential equations}, World Scientific, 1996.

\bibitem{LL} Y.Y. Li, L. Nirenberg, {\it On the Hopf Lemma,} arXiv:0709.3531V1, 2007.

\bibitem{HRY1}R.L. Huang, Y.H. Ye, {\it A convergence result on the second boundary value problem for parabolic equations} arXiv:1806.00431, 2018.


\end{thebibliography}
\end{document}